\documentclass{article}
\usepackage[margin=1in]{geometry}
\PassOptionsToPackage{numbers, compress}{natbib}

\usepackage[utf8]{inputenc} 
\usepackage[T1]{fontenc}    
\usepackage{hyperref}       
\usepackage{url}            
\usepackage{booktabs}       
\usepackage{amsfonts}       
\usepackage{nicefrac}       
\usepackage{microtype}      
\usepackage{xcolor}         
\usepackage{enumitem}       
\usepackage{authblk}

\usepackage{subcaption}
\usepackage{tikz}
\usetikzlibrary{arrows.meta,positioning,calc}
\usepackage{amsmath,amsthm,amssymb}

\newcommand{\sqd}{\mathsf{dev}}

\newcommand{\diag}{\mathsf{diag}}
\newcommand{\diff}{\mathsf{diff}}
\newcommand{\Var}{\mathbb{V}}  
\newcommand{\tr}{\mathsf{tr}}

\newtheorem{theorem}{\bf Theorem}[section]

\newtheorem{definition}[theorem]{\bf Definition}
\newtheorem{observation}[theorem]{\bf Observation}
\newtheorem{proposition}[theorem]{\bf Proposition}

\newtheorem{remark}[theorem]{\bf Remark}

\title{Measuring Spatial Clustering via Metropolis-Hastings Diffusion Distance}

\author[1]{Thomas Weighill}
\author[1]{Chidinma Williams}
\affil[1]{Department of Mathematics and Statistics, University of North Carolina at Greensboro, Greensboro, NC 27402}

\begin{document}

\maketitle

\begin{abstract}

We propose a novel measure of the discrepancy between two probability distributions $f$ and $g$ on a graph -- which we call the diffusion distance -- that measures the rate of convergence of $f$ to $g$ under a graph-constrained Markov chain with stationary distribution $g$. As a default choice for this Markov chain, we use the Metropolis-Hastings transition matrix targeting $g$ with proposals given by a random walk on the graph. Our primary case of interest is when the second distribution $g$ is uniform, in which case the diffusion distance becomes a measure of spatial clustering in $f$. Used in this way, (Metropolis-Hastings) diffusion distance to uniformity extends Moran's $I$-type measures of spatial autocorrelation by incorporating global graph geometry rather than just local patterns. Indeed, Moran's $I$, the most well-known measure of spatial autocorrelation, can be viewed as a one-step heuristic for diffusion distance, so long as specific spatial weights are used. We establish theoretical bounds and a stability result for our measure, connecting it to graph spectra and optimal transport. We then turn our attention to outlining a statistical test for spatial clustering using diffusion distance. Under permutation null models, we derive high-probability bounds on diffusion distance underpinned by exact spectral formulas for convergence of distributions, enabling an efficient statistical test for spatial clustering on large datasets. We empirically compare diffusion distance to Moran's $I$ both as a numerical measure and as a statistical test. We show that diffusion distance exhibits higher power on synthetic data using a stochastic block model. Empirical analysis of Black population distributions for 100 U.S. cities shows that diffusion distance detects subtle differences in urban segregation patterns that Moran's $I$ does not. 
\end{abstract}

\section{Introduction}

A fundamental phenomenon in the analysis of spatial data is that data values can be influenced by their location in relation to other data values. This phenomenon pervades ecology, geography, economics, public health, and urban planning, among other fields. Ignoring this influence and assuming spatial independence can inflate Type I error rates, yield biased regression estimates, and lead to erroneous interpretations of variable relationships~\cite{Legendre1993,LeSage1998}. The most common form of spatial dependence studied in the statistics literature is the dependency of data between adjacent (or very nearby) locations. Spatial autocorrelation quantifies this dependency by measuring correlations between neighboring locations~\cite{Anselin1988}. Positive autocorrelation occurs when similar values cluster together, while negative autocorrelation is associated with checkerboard-like arrangements ~\cite{CliffOrd1973, Getis2008}. Quantifying spatial autocorrelation enables more accurate analyses and better-informed policy decisions across diverse applications~\cite{Anselin1988,Pfeiffer2008}. Scalar indices like Moran's $I$~\cite{Moran1948} are a common way to quantify or test for spatial autocorrelation.

Focusing on the relationship between neighboring data values, however, restricts attention only to local patterns. Large scale patterns, such as the interplay between the global connectivity of the underlying geography and the data, can go undetected unless spatial weights are specifically calibrated to incorporate the correct scale of interactions. In this paper, we address this by introducing a new method to detect spatial dependence in data. Our jumping off point is the observation in ~\cite[Theorem 7.1]{DuchinMurphyWeighill} that Moran's $I$ with a particular choice of spatial weights corresponds to how close the data is to uniformity after one step of a Markov chain. This immediately suggests an expansion of Moran's $I$ -- namely, measuring convergence to uniformity over longer runs of a suitable Markov chain. In other words, we model spatial attributes as probability distributions evolving under natural Markov chain dynamics on graphs, and use convergence speed as a measure of spatial clustering. 

Beginning in the most general context, we introduce the diffusion distance $\diff_{p,\epsilon}(f \to g)$, which measures the minimal steps required for a source distribution $f$ on a graph $G$ to converge to within $\ell_p$ tolerance $\epsilon$ of a target distribution $g$ under a Markov transition matrix designed to target $g$ as stationary distribution. A natural choice for this transition matrix, which we use throughout, is the Metropolis-Hastings transition matrix whose proposal probabilities are given by a random walk on $G$. Our main case of interest is when $g$ is uniform (or captures uniformity in some way). In this setting, slow mixing from $f$ reflects strong spatial clustering in $f$, while fast convergence indicates near-uniformity of $f$. Notably, diffusion distance depends on the global geometry of the underlying graph, rather than just local correlations. 

We begin with the theoretical study of diffusion distance and its connections to other important concepts in spatial data analysis. We establish bounds in terms of graph spectra, and in terms of Wasserstein distance on graphs. We prove a stability result controlling the effect of small changes in $f$ on $\diff_{p,\epsilon}(f \to g)$. We also make the connection to Moran's $I$ mentioned above explicit.

In order to use diffusion distance (to uniformity) for statistical tests, we require expectations under various null models similar to those for Moran's $I$. Under a random permutation null model, we derive the exact expectation and variance of the $n$-step reduction in squared deviation, resulting in a formula for bounding the \textit{p}-value of an observed diffusion distance in terms of the spectrum of the underlying graph. This allows us to avoid sampling a prohibitive number of random permutations to obtain a \textit{p}-value for large datasets. We empirically demonstrate the power of the diffusion distance test on synthetic graphs generated using a stochastic block model (SBM).

Beyond statistical tests, we also demonstrate diffusion distance as a way to quantify spatial clustering patterns. We compute diffusion distances for 100 U.S.~cities using Black population distributions as our data of interest; in this context, spatial clustering corresponds to spatial sorting and/or segregation. We compare to Moran's $I$ as a benchmark, and highlight specific instances where the two measures agree and diverge respectively.

\subsection{Related work}
\label{sec:related}

\paragraph{Other measures of spatial autocorrelation.}
The dominant tools for quantifying spatial dependence are Moran's $I$, Geary's $C$, and the family of local indicators of spatial association (LISA)~\cite{Moran1950,Geary1954,Anselin1995}. These statistics enjoy clear interpretations and closed-form inference under various null models, making them the standard benchmarks in spatial analysis. Moran's $I$ serves as our benchmark in this paper, and can be interpreted as a one-step heuristic for diffusion distance (Theorem~\ref{thm:moran-sqd}). Our approach departs from these classical statistics by measuring diffusion under global graph dynamics rather than local correlations. A broad literature enriches Moran's $I$-type scalar summaries by capturing spatial heterogeneity through eigenvector spatial filtering, geographically weighted regression, Bayesian spatial models, and network-based methods~\cite{CliffOrd1973,CliffOrd1981,Tiefelsdorf1998,Fotheringham2002, Gelfand2003,Griffith2003,Chun2016,Thayn2017,Fotheringham2017}. Our method adds to this literature by introducing a measure which is readily interpretable in terms of graph dynamics and which has closed-form theoretical results under random permutation models.


\paragraph{Spectral graph theory.}
The study of graphs via the spectra of graph Laplacians (and other operators) is well-established and has been deployed in many data science applications such as spectral clustering, diffusion maps~\cite{CoifmanLafon2006}, and graph diffusion distance~\cite{hammond2013graph, vayer2019optimal} (a measure of similarity between two graphs, and not a distance between distributions on a common graph as defined in this paper). Graph spectra are also intimately related via classical results to mixing times for random walks. We exploit this connection to prove results about convergence from a specific initial distribution $f$ rather than worst-case mixing bounds, so that our measure captures not only information about a graph, but also about how the given distribution $f$ interacts with the global geometry of that graph (see Theorem~\ref{thm:diff-upper-bound-ergodic}).

 \paragraph{Segregation measurement.}
Quantifying racial and socioeconomic segregation in cities has motivated a rich methodological literature~\cite{massey1988dimensions,reardon2004measures}. Classical indices such as the dissimilarity index and isolation index capture global separation but are insensitive to the spatial arrangement of groups~\cite{wong2003spatial}. Spatially explicit segregation measures attempt to incorporate geographic contiguity and local mixing~\cite{reardon2004measures,wong2003spatial,knaap2020segregation}. Our framework contributes to this literature by providing a mixing-time interpretation of segregation: a city is more segregated insofar as its demographic distribution takes longer to equilibrate under local graph diffusion. This reinterpretation connects segregation measurement directly to the spectral geometry of the city's adjacency graph.
Our work is similar in spirit to measures of segregation that attempt to measure the propensity of members of one class to encounter another class~\cite{van2014random, sousa2022quantifying}, but there are qualitative differences. For example, locally heterogenous population distributions with a global tilt result in many encounters, which may hide the overall trend. This might be less of an issue when measuring segregation, but can be a problem for statistical inference.

\paragraph{Optimal transport and geometry-aware comparison.}
Optimal transport provides a geometry-aware way to compare probability distributions by quantifying the cost of redistributing mass~\cite{Peyre2016,Villani2008}. Recent spatial applications include spatio-temporal transport~\cite{janati2020spatio}, and geometry-aware segregation measures~\cite{PengMurphy2025}. Our approach is complementary: rather than asking how costly it is to optimally move mass from $f$ to $\tau$, we ask how many steps of graph-constrained diffusion are required. 

\subsection{Limitations}

Our target application is geospatial data, and like almost all geospatial data analysis techniques, our method is vulnerable to the Modifiable Areal Unit Problem (MAUP), wherein reaggregation of data by different geographic units results in different outputs. The fact that our measure includes a diffusion process suggests we are less vulnerable to the ``checkerboard problem'' (a reaggregation of an checkerboard pattern can suddenly become uniform). Another limitation is that we do not attempt to model actual population dynamics, instead relying on random walks on graphs as an approximation, leaving aside many factors such as public transport availability or traffic patterns. Notably, our measure depends heavily on the geometry of geographic adjacency graphs, which is understudied (but see the very recent~\cite{anderson2026census}). 


\section{Preliminaries}
\label{sec:preliminaries}

We denote by $\Delta_N$ the \emph{probability simplex} $\Delta_N = \{\mathbf{v} \in \mathbb{R}^{N} \mid \forall_i \mathbf{v}_i \geq 0,\ \sum_i \mathbf{v}_i = 1\}$. Since all our state spaces are finite, probability distributions are elements of $\Delta_N$. Vectors will be row vectors throughout this paper, and we denote by $\mathbf{1}$ the all-ones vector of the appropriate size. We denote the uniform distribution $\frac{1}{N} \mathbf{1}$ by $\tau$ throughout. A matrix $M$ is \emph{row-stochastic} if $M \mathbf{1}^\top = \mathbf{1}^\top$, and \emph{bi-stochastic} if in addition $\mathbf{1}M = \mathbf{1}$.

\subsection{Spatial autocorrelation}
\label{subsec:moran_i}

\emph{Moran's $I$} is a classical global measure of spatial autocorrelation that quantifies whether nearby observations tend to have similar values~\cite{Moran1948,Anselin1995}. Let $\mathcal{V}=\{v_1,\dots,v_N\}$ denote observations with mean $\bar v=\frac{1}{N}\sum_{i=1}^N v_i$, and let $W=[w_{ij}]$ be a nonnegative spatial weight matrix. Then Moran's $I$ is
\begin{equation}
I
=
\frac{N}{\sum_{i=1}^N\sum_{j=1}^N w_{ij}}
\cdot
\frac{\sum_{i=1}^N\sum_{j=1}^N w_{ij}(v_i-\bar v)(v_j-\bar v)}
{\sum_{i=1}^N (v_i-\bar v)^2}.
\label{eq:morans_i}
\end{equation}
Positive values indicate positive spatial autocorrelation, negative values indicate dispersion, and values near zero indicate weak global structure~\cite{Anselin1995}. Under the randomization null and traceless spatial weights matrix,
$
\mathbb{E}[I] = -\frac{1}{N-1},
$
providing the classical baseline for spatial randomness~\cite{CliffOrd1981,MoodGraybillBoes1974}. A common default choice for the spatial weights, which we use throughout the paper, are row-standardized weights based on an adjacency graph: 
$$w_{ij} = \begin{cases}
    \frac{1}{\deg i} & \text{$i$ and $j$ are adjacent and $i \neq j$} \\
    0 & \text{otherwise}
\end{cases}$$ 

\subsection{Markov chains and Metropolis-Hastings}

Let $P\in\mathbb{R}^{N\times N}$ be a row-stochastic transition matrix on a finite state space with $N$ states. For a distribution $f\in \Delta_N$ modeled as a row vector, the distribution after $n$ steps is $fP^n$. A distribution $\tau$ is \emph{stationary for $P$} if $\tau=\tau P$. When $P$ is irreducible and aperiodic on a finite state space, it admits a unique stationary distribution $\tau$, and $fP^n\to\tau$ as $n\to\infty$ for any initial distribution $f$. Given a transition matrix $Q$ (referred to as a \emph{proposal} chain) and a target distribution $g$, we can define the \emph{Metropolis-Hastings Markov chain} $P$ whose transition probability from $i$ to $j \neq i$ is given by 
$$
P_{ij} = \min \left(Q_{ij}, \frac{Q_{ji} g_j}{g_i} \right).
$$
Practically speaking, this amounts to proposing a new state using $Q$ and accepting with probability proportional to the proposed change in $g$. By design, $g$ is stationary for $P$. Metropolis-Hastings Markov chains were originally designed to sample from complicated distributions known only up to a constant factor. We borrow this construction for a different purpose, namely to construct a canonical Markov chain with specified equilibrium distribution. Note that by~\cite[Theorem 1]{billera2001geometric}, we can view the Metropolis-Hastings Markov chain as an optimal approximation to the proposal chain $Q$ with specified stationary distribution. 

\subsection{Optimal transport and total variation distance}
Given probability measures $\mu=\sum_{i=1}^n a_i\delta_{v_i}$ and $\nu=\sum_{j=1}^m b_j\delta_{u_j}$ on a finite metric space $(X,d)$, the $p$-Wasserstein distance measures the cost to move $\mu$ to $\nu$, and is given by
\[
W_p(\mu,\nu)
=
\left(
\min_{\pi\in\Pi(\mu,\nu)}
\sum_{i=1}^n\sum_{j=1}^m d(v_i,u_j)^p\,\pi_{ij}
\right)^{1/p},
\]
where $\Pi(\mu,\nu)$ denotes the set of couplings with marginals $\mu$ and $\nu$~\cite{Kantorovich1939,Vaserstein1969}. Recall that the total variation distance between $\mu$ and $\nu$ is ,
$
\|\mu-\nu\|_{\mathrm{TV}}
=
\frac12\sum_{i=1}^N |\mu(v_i)-\nu(v_i)|.
$
On spaces of diameter $D$, these quantities satisfy the following:
$
\label{eq:wass-tv-diameter}
W_1(\mu,\nu)\le 2D\,\|\mu-\nu\|_{\mathrm{TV}}
$ (see \cite{Villani2008}).

\section{(Metropolis-Hastings) diffusion distance}
\label{sec:method}

We begin with the general concept of the diffusion distance from one probability distribution to another on a common state space.

\begin{definition}[Diffusion distance in $\ell_p$-norm]
\label{def:diff-distance}
Let $f, g \in \Delta_N$ be probability distributions on a finite state space, let $P$ be a
Markov transition matrix, let $1 \le p < \infty$, and let $\epsilon > 0$.
The \emph{$\ell_p$-diffusion distance} from $f$ to $g$ is
\[
  \diff^P_{p,\epsilon}(f \to g)
  := \min\bigl\{ n \in \mathbb{N}_0 : \|fP^n - g\|_p < \epsilon \bigr\}.
\]
\end{definition}

Despite the name, diffusion distance is not a metric. It is assymetric in general, as suggested by the notation $f \to g$, and for certain choices of $g$ and $P$ can take the value $\infty$. The parameter $\epsilon$ is a tolerance which must be chosen ahead of time (see Remark~\ref{rem:heuristic} for some recommendations).

For practical purposes we are mostly interested in the following special case (in which case we suppress the $P$ in $\diff^P$):
\begin{itemize}
    \item $f$ is a probability distribution on a graph $G$ and $g = \tau$ is uniform,
    \item $P$ is the Metropolis-Hastings Markov chain targeting $\tau$ whose proposal chain is a random walk on $G$.
\end{itemize}
Note that $P$ is symmetric and bistochastic in this special case. We refer to this setup as \emph{Metropolis-Hastings diffusion distance to uniformity}. Note that when $G$ is connected, $P$ is ergodic as soon as $G$ is non-bipartite (e.g.~contains a triangle), and thus we assume throughout that this is the case.

\begin{remark}[Heuristics for choosing $p$ and $\epsilon$]\label{rem:heuristic}
Metropolis-Hastings diffusion distance to uniformity depends on the choice of $p$ and $\epsilon$. While there may be case-specific reasons for a particular choice, we give some general recommendations here. 

For statistical tests, the underlying graph is typically fixed, and the best theoretical results and approximations are available for $p = 2$ (Theorem~\ref{thm:hp-normal-permutation}). A reasonable way to pick $\epsilon$ for a graph with $N$ nodes is to note that $\mathbb{E}[||f - \tau||_2] \approx \frac{1}{\sqrt{N}}$ for large $N$ where $f$ is drawn uniformly from the simplex $\Delta_N$, and set $\epsilon$ to be about a tenth of this value. This ensures that even spatially independent data will typically have positive diffusion distance, creating a strong right-tailed test. 

For measuring clustering directly, the graph size may vary, and $p = 1$ has better scale invariance. In particular, note that a distribution $f$ which is zero on a fraction $x \in [0,1]$ of the nodes of a graph and uniform on the rest satisfies $||f-\tau||_1 = 2x$. We can thus visualize $\epsilon$ as the $\ell^1$ deviation from uniform achieved when the fraction $\epsilon / 2$ of the graph is uniformly covered, regardless of the size of the graph, which can help us calibrate $\epsilon$.

\end{remark}


\subsection{Theoretical bounds}
\label{subsec:basic-properties}

We now record some theoretical properties of diffusion distance. Each is stated in a convenient level of generality, but note that all results in this section apply to Metropolis-Hasting diffusion distance to uniformity. 

\begin{theorem}[Spectral upper bound for ergodic chains targeting uniformity]
\label{thm:diff-upper-bound-ergodic}
Let $P$ be the transition matrix of an irreducible, aperiodic Markov chain on $N$ states with stationary distribution $\tau$, let $1 \le p < \infty$, and assume $P$ is diagonalizable over $\mathbb{C}$. Let $1 = \lambda_1, \lambda_2, \dots, \lambda_N$ be the eigenvalues of $P$, with right eigenvectors $\{u_i\}_{i=1}^N$ and left eigenvectors $\{v_i^\top\}_{i=1}^N$ satisfying
\[
  Pu_i = \lambda_i u_i, \qquad
  v_i^\top P = \lambda_i v_i^\top, \qquad
  u_1 = \mathbf{1}, \qquad
  v_1^\top = \tau,
\]
and the biorthogonality relations $v_i^\top u_j = \delta_{ij}$, where
$\delta_{ij} = 1$ if $i = j$ and $0$ otherwise. For $i \ge 2$, normalize
so that $\|v_i^\top\|_p = 1$, and set $
  \lambda := \max_{2 \le i \le N}|\lambda_i| < 1$.
Write any initial distribution $f$ with $f\mathbf{1} = 1$ as
$f = \sum_{i=1}^N c_i v_i^\top$ with $c_i = fu_i$. If $f = \tau$ then
$\diff^P_{p,\epsilon}(f \to \tau) = 0$; otherwise,
\[
  \diff^P_{p,\epsilon}(f \to \tau)
  \;\le\;
  \left\lceil
    \frac{\log\!\bigl(\sum_{i=2}^N |c_i|\,/\,\epsilon\bigr)}{\log(1/\lambda)}
  \right\rceil.
\]
\end{theorem}
\begin{proof}
\label{proof:thm:diff-ergodic}
Since $P$ is irreducible and aperiodic, it is primitive: some power of
$P$ has all strictly positive entries~\cite{meyer2000matrix}. The
Perron--Frobenius theorem then guarantees that the eigenvalue $1$ is
simple and that $|\lambda_i| < 1$ for every $i \ge 2$~\cite{meyer2000matrix};
in particular $\lambda < 1$.

Because $\{v_i^\top\}_{i=1}^N$ is a basis of left eigenvectors,
every row vector $f$ with $f\mathbf{1} = 1$ has a unique expansion
$f = \sum_{i=1}^N c_i v_i^\top$. Multiplying on the right by $u_j$
and applying biorthogonality gives $c_j = fu_j$ for each $j$.
In particular $c_1 = f\mathbf{1} = 1$, so since $v_1^\top = \tau$,
\[
  f - \tau = \sum_{i=2}^N c_i v_i^\top.
\]
Using $v_i^\top P^n = \lambda_i^n v_i^\top$ (which follows by induction
from the left eigenvector relation $v_i^\top P = \lambda_i v_i^\top$)
together with $\tau P^n = \tau$,
\[
  fP^n - \tau = (f - \tau)P^n = \sum_{i=2}^N c_i \lambda_i^n v_i^\top.
\]
By the triangle inequality and $\|v_i^\top\|_p = 1$,
\[
  \|fP^n - \tau\|_p
  \;\le\; \sum_{i=2}^N |c_i|\,|\lambda_i|^n
  \;\le\; \lambda^n \sum_{i=2}^N |c_i|,
\]
where the last step uses $|\lambda_i| \le \lambda$ for all $i \ge 2$.

If $f = \tau$, then $c_i = \tau u_i = v_1^\top u_i = \delta_{1i}$
for all $i$ by biorthogonality, so $\sum_{i=2}^N |c_i| = 0$ and
$fP^n = \tau$ for all $n \ge 0$, giving
$\diff^P_{p,\epsilon}(f \to \tau) = 0$. Otherwise, set
$S := \sum_{i=2}^N |c_i| > 0$. The bound $\|fP^n - \tau\|_p \le
\lambda^n S$ falls below $\epsilon$ as soon as $\lambda^n S < \epsilon$,
which is equivalent to
\[
  n > \frac{\log(S/\epsilon)}{\log(1/\lambda)}.
\]
The smallest such integer $n$ is $\lceil \log(S/\epsilon)/\log(1/\lambda)
\rceil$, which gives the stated upper bound on $\diff^P_{p,\epsilon}(f \to
\tau)$.
\end{proof}

\begin{proposition}[Lower bound in terms of Wasserstein distance]
\label{prop:wasserstein-bound}
Let $G = (V, E)$ be a graph with shortest-path metric $d$, diameter
$D$, and unit edge lengths. Let $P$ be a Markov transition matrix with
$P_{ij} > 0$ only if $(i, j) \in E$, with stationary distribution $g$. Then for a probability measure $f$
on $V$,
\[
  W_1(f, g) - D\epsilon \le  \diff^P_{1,\epsilon}(f \to g).
\]
\end{proposition}
\begin{proof}
Let
\[
  n := \diff^P_{1,\epsilon}(f \to g).
\]
By definition of diffusion distance,
\[
  \|fP^n - g\|_1 < \epsilon.
\]
Using~\eqref{eq:wass-tv-diameter},
\[
  W_1(fP^n, g)
  \le 2D\,\|fP^n - g\|_{\mathrm{TV}}
  = D\,\|fP^n - g\|_1
  < D\epsilon.
\]
By the triangle inequality,
\[
  W_1(f,g)
  \le W_1(f,fP^n) + W_1(fP^n,g).
\]
It therefore remains to bound \(W_1(f,fP^n)\).

Consider the matrix
\[
  \Gamma := \operatorname{diag}(f)\,P^n.
\]
Its \((i,j)\) entry is
\[
  \Gamma_{ij} = f_i (P^n)_{ij}.
\]
Since \(P^n\) is row-stochastic, the row sums of \(\Gamma\) satisfy
\[
  \sum_j \Gamma_{ij}
  = f_i \sum_j (P^n)_{ij}
  = f_i,
\]
and the column sums satisfy
\[
  \sum_i \Gamma_{ij}
  = \sum_i f_i (P^n)_{ij}
  = (fP^n)_j.
\]
Thus \(\Gamma\) is a coupling of \(f\) and \(fP^n\), and therefore
\[
  W_1(f,fP^n)
  \le \sum_{i,j} d(i,j)\,\Gamma_{ij}.
\]

Now, if \((P^n)_{ij} > 0\), then there exists a path from \(i\) to \(j\)
of length at most \(n\) using edges along which \(P\) assigns positive
transition probability. Since \(P_{ij} > 0\) only if \((i,j)\in E\) and
each edge has unit length, it follows that \(d(i,j) \le n\) whenever
\(\Gamma_{ij} > 0\). Hence
\[
  W_1(f,fP^n)
  \le \sum_{i,j} n\,\Gamma_{ij}
  = n \sum_{i,j} \Gamma_{ij}
  = n.
\]
Combining the bounds gives
\[
  W_1(f,g)
  < n + D\epsilon.
\]
Equivalently,
\[
  W_1(f,g) - D\epsilon < n
  = \diff^P_{1,\epsilon}(f \to g).
\]
Since \(\diff^P_{1,\epsilon}(f \to g)\) is an integer, this implies
\[
  W_1(f,g) - D\epsilon \le \diff^P_{1,\epsilon}(f \to g),
\]
as claimed.
\end{proof}

We close with a stability result controlling the effect of perturbations of the source distribution. In particular, we prove that the diffusion distance from a function $g$ which is close to $f$ is bounded above and below by the diffusion distance of $g$ with shifted tolerances.

\begin{proposition}[Stability from source distribution to uniform]
\label{prop:stability-initial}
Let $P$ be a bistochastic Markov transition matrix on $N$ states with uniform stationary distribution $\tau$, and let $1 \le p < \infty$. Let $f, g \in \Delta_N$ be probability distributions such that $||f-g||_p <\delta$. Then for any $\epsilon > \delta$,
\[
  \diff^P_{p,\epsilon + \delta}(f \to \tau) \le \diff^P_{p,\epsilon}(g \to \tau) \leq \diff^P_{p,\epsilon - \delta}(f \to \tau)
\]
\end{proposition}
\begin{proof}
Only the left inequality requires proof; the right follows by symmetry, after replacing $\epsilon$ with $\epsilon - \delta$. Let
\[
  n := \diff^P_{p,\epsilon}(g \to \tau).
\]
By definition of diffusion distance,
\[
  \|gP^n - \tau\|_p < \epsilon.
\]
Using the triangle inequality,
\begin{align*}
  \|fP^n - \tau\|_p
  &= \|fP^n - gP^n + gP^n - \tau\|_p \\
  &\le \|(f-g)P^n\|_p + \|gP^n - \tau\|_p.
\end{align*}
Since $P$ is bistochastic, the Birkhoff--von Neumann theorem~\cite{Birkhoff1946} expresses $P$ as a convex combination of permutation matrices:
\[
  P = \sum_{k=1}^K \alpha_k \Pi_k,
  \qquad
  \alpha_k \ge 0,
  \qquad
  \sum_{k=1}^K \alpha_k = 1,
\]
where each $\Pi_k$ is a permutation matrix. For any $x \in \mathbb{R}^N$, each permutation preserves the $\ell_p$ norm, so
\[
  \|x\Pi_k\|_p = \|x\|_p.
\]
Hence
\begin{align*}
  \|xP\|_p
  = \left\|\sum_{k=1}^K \alpha_k x\Pi_k\right\|_p 
  \le \sum_{k=1}^K \alpha_k \|x\Pi_k\|_p 
  = \sum_{k=1}^K \alpha_k \|x\|_p 
  = \|x\|_p.
\end{align*}
Applying this inductively gives
\[
  \|xP^n\|_p \le \|x\|_p
  \qquad \text{for all } n \ge 0.
\]
Therefore,
\[
  \|(f-g)P^n\|_p \le \|f-g\|_p < \delta.
\]
Combining the bounds,
\[
  \|fP^n - \tau\|_p
  < \delta + \epsilon.
\]
By definition of diffusion distance, this implies
\[
  \diff^P_{p,\epsilon+\delta}(f \to \tau)
  \le n
  =
  \diff^P_{p,\epsilon}(g \to \tau),
\]
as claimed.
\end{proof}

\subsection{Connection to Moran's \texorpdfstring{$I$}{I}}
\label{sec:moranconnection}

To study one-step behavior, we use the \emph{squared deviation} from a reference distribution $g$,
\begin{equation}
\label{eq:sqd_definition}
  \sqd_g(f) := \frac{1}{N}\|f - g\|_2^2.
\end{equation}
The one-step squared deviation $\sqd_g(fP)$ provides the link between our framework to Moran-type spatial autocorrelation statistics.

\begin{theorem}[Moran's $I$ as normalized one-step squared deviation~\cite{DuchinMurphyWeighill}]
\label{thm:moran-sqd}
Let $P$ be a bistochastic Markov matrix on $N$ states with uniform stationary distribution $\tau$, let $f \in \Delta_N$ be a probability vector with $\sqd_\tau(f) > 0$, and set $M := PP^\top$. 
Then 
\[
\label{eq:moran-sqd-identity}
  I_\tau(f, M) = \frac{\sqd_\tau(fP)}{\sqd_\tau(f)}.
\]
\end{theorem}

\begin{proof}
Since $P$ is bistochastic, $M\mathbf{1} = PP^\top\mathbf{1} = P\mathbf{1} = \mathbf{1}$, so every row of $M$ sums to $1$ and
\[
  S_0 = \sum_{i,j=1}^N M_{ij} = N.
\]
Substituting $S_0 = N$ into the definition of $I_\tau(f,M)$,
\[
  I_\tau(f, M)
  = \frac{(f-\tau)^\top M(f-\tau)}{\|f-\tau\|_2^2}.
\]
Dividing the definitions of $\sqd_\tau(fP)$ and $\sqd_\tau(f)$,
\[
  \frac{\sqd_\tau(fP)}{\sqd_\tau(f)}
  = \frac{\tfrac{1}{N}(f-\tau)^\top M(f-\tau)}
         {\tfrac{1}{N}\|f-\tau\|_2^2}
  = \frac{(f-\tau)^\top M(f-\tau)}{\|f-\tau\|_2^2}
  = I_\tau(f, M).
\]
\end{proof}

Diffusion distance therefore expands Moran's $I$ in two ways. Firstly, the usual spatial weights are replaced by the Markov chain-inspired weights $PP^\top$ with $P$ a transition matrix, as first proposed in \cite{DuchinMurphyWeighill}. And secondly, multi-step convergence is measured instead of just one-step convergence. This hints at a limitation of Moran's $I$: as a one-step statistic, it cannot detect structural differences that manifest only across multiple steps. Two systems may exhibit identical one-step contraction yet very different long-run convergence behavior, which diffusion distance is designed to capture. Figure~\ref{fig:3examples} shows an illustrative example with grid (sub)graphs.

\begin{figure}
    \centering
    \begin{subfigure}{0.32\textwidth}
        \centering
        \includegraphics[width=\textwidth]{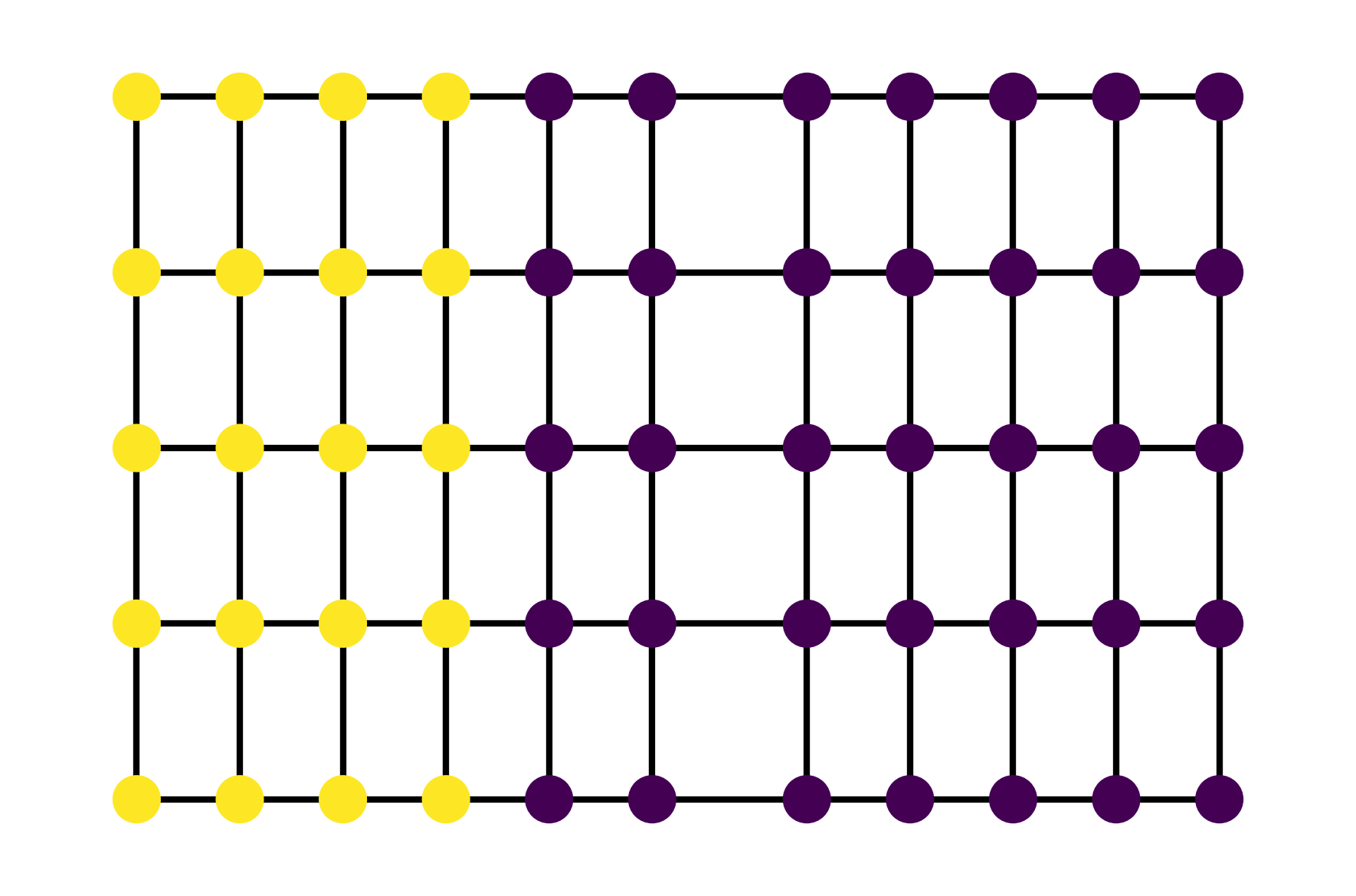}
        \caption{$I = 0.89, \diff_{0.01} = 118$}
    \end{subfigure}
    \begin{subfigure}{0.32\textwidth}
        \centering
        \includegraphics[width=\textwidth]{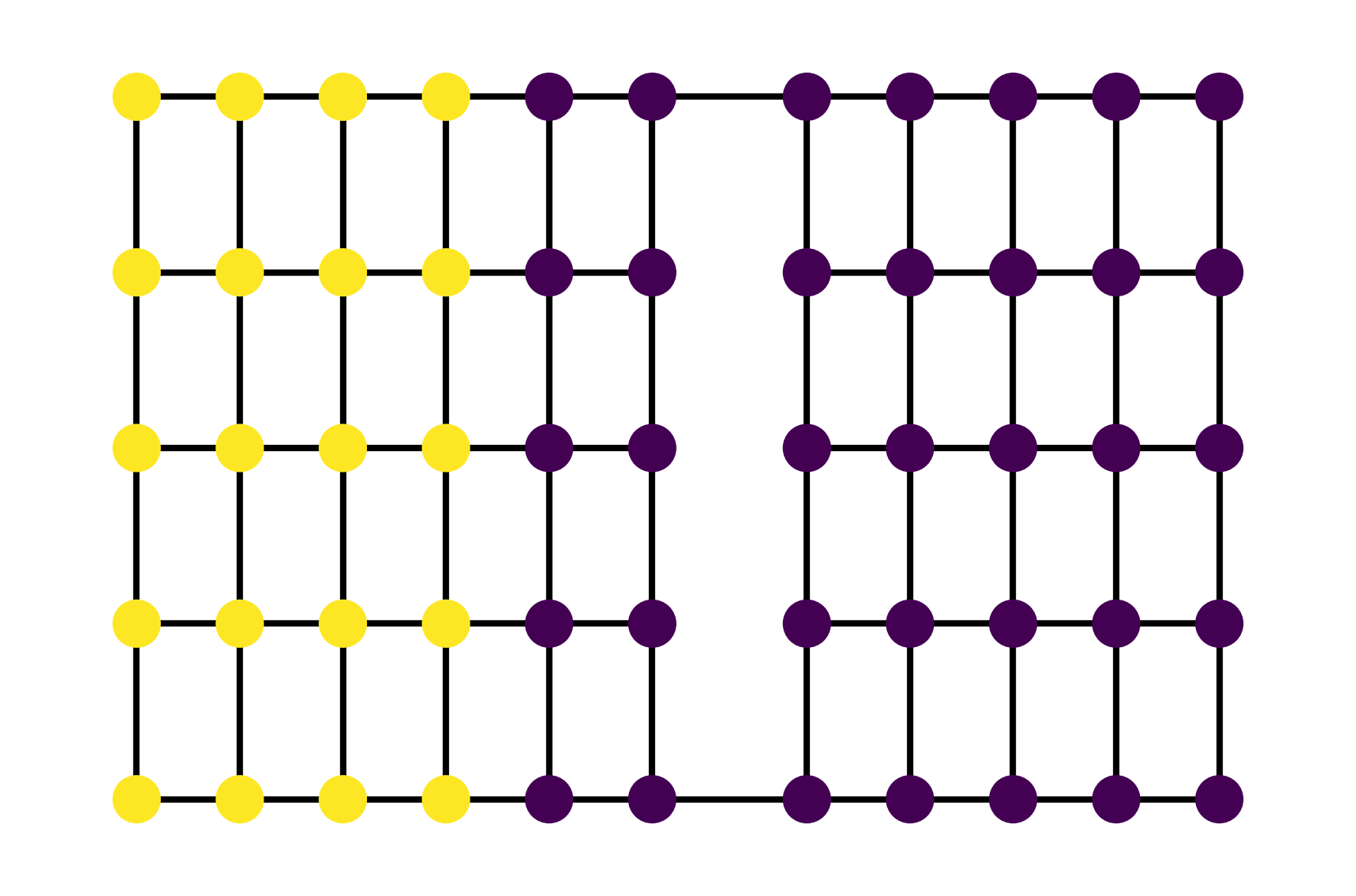}
        \caption{$I = 0.89, \diff_{0.01} = 164$}
    \end{subfigure}
    \begin{subfigure}{0.32\textwidth}
        \centering
\includegraphics[width=\textwidth]{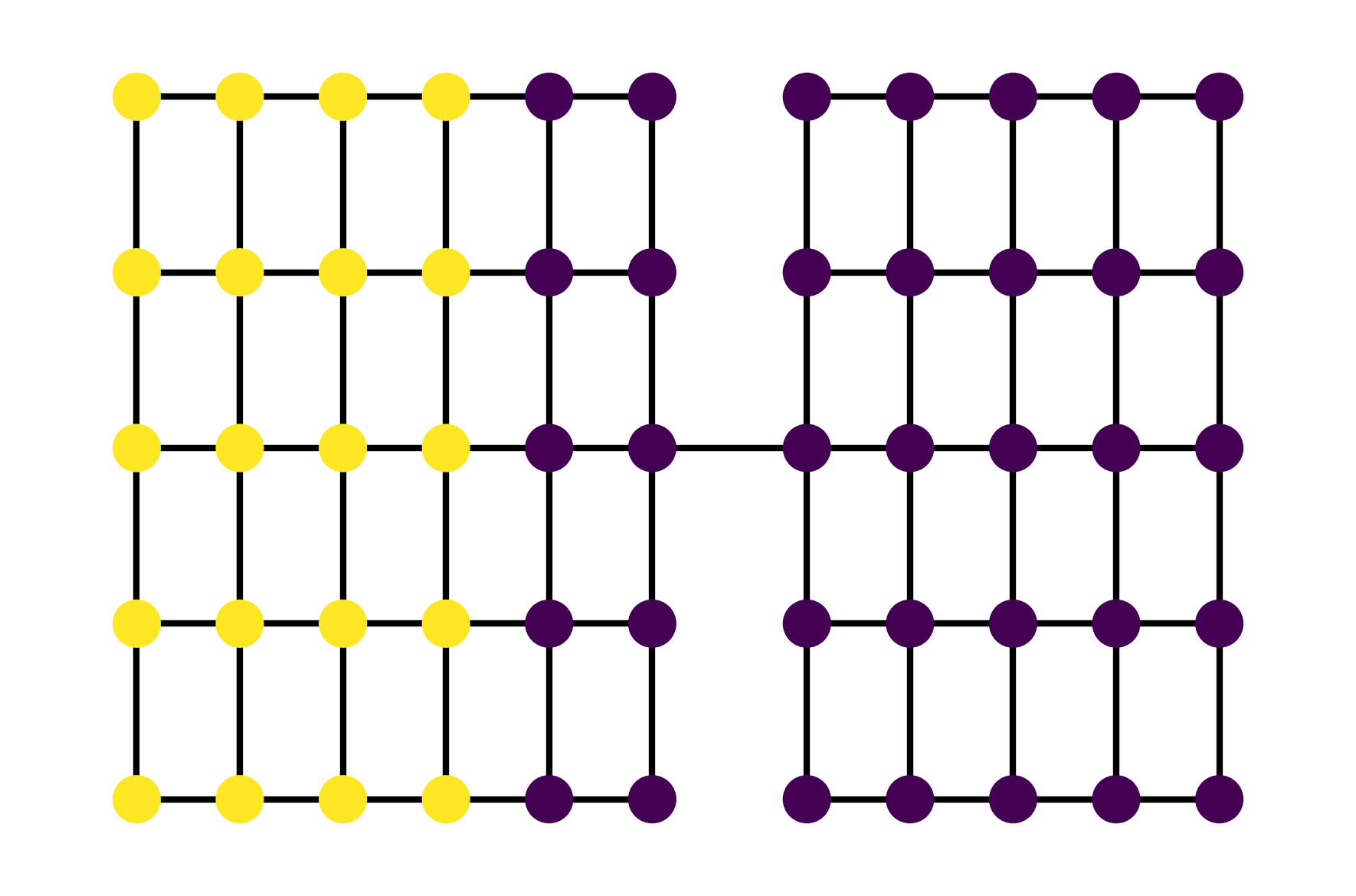}
\caption{$I = 0.89, \diff_{0.01} = 301$}
    \end{subfigure}
    \caption{Three graphs with uniform distribtion $f$ supported on the yellow nodes. The graphs have equal diameters and equal Moran's $I$ for $f$, but different diffusion distances to uniform with $p = 2$, $\epsilon = 0.01$ due to the change in geometry.}
    \label{fig:3examples}
\end{figure}

\section{Inference under permutation null models}
\label{sec:statistics}

In addition to quantifying the extent of spatial clustering, Metropolis-Hastings diffusion distance to uniformity can be used as a statistical test for the presence of spatial clustering. To enable this, we develop inference under permutation null models. Given a probability distribution $f \in \Delta_N$, the permutation null model is the uniform distribution on all permutations of the values of $f$. If $\pi$ is a permutation of $\{1,\ldots, N\}$, then we denote by $f_\pi$ the vector $f$ with values rearranged by $\pi$. Permuting spatial labels disrupts location dependence while preserving the empirical mass values, and thus provides a natural baseline for assessing whether an observed spatial pattern is unusually clustered relative to random relabelings on the same graph. It is also a common null model used in applications of Moran's $I$~\cite{zhang2007decomposition}, since it does not assume any particular form of the distribution at each location. We analyze one-step and $n$-step squared deviations, for which null moments can be computed exactly.

\subsection{Squared deviation under the null}
\label{subsec:sqd-null}

Let $\pi$ be a uniformly random permutation of $\{1,\dots,N\}$. The
following theorem collects the first and second moments of the permutation
null distribution of $\sqd(f_\pi P)$. 

\begin{theorem}[Permutation moments of squared deviation, uniform stationary distribution]
\label{thm:one-step-moments-uniform}
Assume $N \ge 4$. Let $P$ be an $N \times N$ transition matrix with uniform stationary distribution $\tau = \frac{1}{N}\mathbf{1}$, and set $M := PP^\top$. Let $\pi$ be uniformly distributed over all permutations of $\{1,\dots,N\}$, let $f$ be a probability vector, and define
\[
  Z_a := f_a - \tfrac{1}{N}, \qquad
  \sqd(f) := \tfrac{1}{N}\sum_{a=1}^N Z_a^2, \qquad
  S_4(Z) := \sum_{a=1}^N Z_a^4, \qquad
  \Delta := \sum_{i=1}^N M_{ii}^2.
\]
Writing $\mathcal{D} := (N-1)(N-2)(N-3)$, the first moment is
\[
  \mathbb{E}_\pi\bigl[\sqd(f_\pi P)\bigr]
  = \sqd(f)\,\frac{\tr(M) - 1}{N - 1},
\]
and the second moment is
\begin{align*}
  \mathbb{E}_\pi\bigl[\sqd(f_\pi P)^2\bigr]
  &= \frac{\sqd(f)^2}{N\mathcal{D}}
  \Bigl[
    (N^2-3N+3)\bigl(\tr(M)^2 + 2\tr(M^2)\bigr) - N^2 \\
  &\hspace{2.7cm}
    - 2(N^2-6N+6)\tr(M) - 3N(N-1)\Delta
  \Bigr] \\
  &\quad
  + \frac{S_4(Z)}{N^2\mathcal{D}}
  \Bigl[
    -(N-1)\bigl(\tr(M)^2 + 2\tr(M^2)\bigr) + 2N \\
  &\hspace{2.7cm}
    - 4\tr(M) + N(N+1)\Delta
  \Bigr],
\end{align*}
and consequently
\[
\Var_\pi\bigl(\sqd(f_\pi P)\bigr) = \mathbb{E}_\pi\bigl[\sqd(f_\pi P)^2\bigr] - \Bigl(\mathbb{E}_\pi\bigl[\sqd(f_\pi P)\bigr]\Bigr)^2.
\]
\end{theorem}

Under the permutation null, we see that the expected one-step squared deviation factors into the initial deviation $\sqd(f)$ and a graph-dependent term controlled by the diagonal of $M = PP^\top$, which encodes the self-similarity of the chain's one-step transitions. The second-moment formula involves the graph-structural quantities $\tr(M)$, $\tr(M^2)$, and $\Delta$. The full proof of Theorem~\ref{thm:one-step-moments-uniform} is outlined in
Appendix~\ref{proof:thm:one-step-moments-uniform}.

\begin{observation}\label{obs:symmetricP}
    Note that when $P$ is symmetric, we can express the key quantities in Theorem~\ref{thm:one-step-moments-uniform} in terms of the eigenvalues $\lambda_i$ and orthonormal eigenvectors $v_i$ of $P$. The following two require only eigenvalues
\begin{align*}
    \tr(M) = \tr(P^2) = \sum_{i} \lambda^2_i, \quad 
    \tr(M^2)= \tr(P^4) = \sum_{i} \lambda_i^4 
    \end{align*}
While the third also requires eigenvectors:
\begin{align*}
\Delta = \sum_{i} \left( P^2_{ii} \right)^2 = \sum_i \left(  \sum_{k}  \lambda_k^2 v_{ik}^2  \right)^2 = \sum_k \sum_\ell \lambda_k^2 \lambda_\ell^2 \sum_i v_{ik}^2 v_{i\ell}^2
\end{align*}
Note that since the $v_k$ are unit vectors, we have $\sum_i v_{ik}^2 v_{i\ell}^2 \leq \sum_i v_{ik}^2 = 1$ for all $k,\ell$. The eigenvector dependence disappears when the diagonal of $M$ is constant: if $M_{ii} = c$ for all $i$, then $\tr(M) = Nc$ and hence $\Delta = \sum_i M_{ii}^2 = Nc^2 = \tr(M)^2/N$. This holds whenever the graph is $D$-regular, where $P_{ij} = 1/D$ on each edge and so $M_{ii} = 1/D$ for every $i$. In that case, all three quantities depend on the eigenvalues of $P$ alone.
\end{observation}

\begin{remark}
    Given that one-step squared deviation is equivalent to Moran's $I$ by Theorem~\ref{thm:moran-sqd}, one might assume that we can reuse classical moment calculations for Moran's $I$. However, these classical results always assume traceless spatial weights matrix (see e.g.~\cite{CliffOrd1981}), which is very different to our case where $M = PP^\top$ typically has positive trace, requiring us to derive new results. Note that if $\mathsf{tr}(M) = 0$ and $\sqd(f) = 1$ then Theorem~\ref{thm:one-step-moments-uniform} recovers the classical expectation of $-\frac{1}{N-1}$.
\end{remark}

\subsection{High-probability control of diffusion distance}
\label{subsec:hp-diff}

For the $n$-step null and the resulting concentration bound, we 
restrict to normal transition matrices $P$, for which the spectral structure of $M_n = P^n(P^n)^\top$ is fully determined by the eigenvalues of $PP^\top$. In particular, this is the case for a Metropolis-Hastings chain targeting uniformity. The proof of the following Theorem uses Cantelli's inequality and the expectations in Theorem~\ref{thm:one-step-moments-uniform}.

\begin{theorem}[High-probability bound on diffusion distance under permutation]
\label{thm:hp-normal-permutation}
Assume $N \ge 4$. Let $P$ be an $N \times N$ normal transition matrix
with uniform stationary distribution $\tau = \frac{1}{N}\mathbf{1}$, let
$f \in \Delta_N$ be a probability distribution, and let $\pi$ be a uniformly random
permutation of $\{1,\dots,N\}$. For $n \ge 0$, define
\[
  E_n := \mathbb{E}_\pi\bigl[\sqd(f_\pi P^n)\bigr], \qquad
  V_n := \Var_\pi\bigl(\sqd(f_\pi P^n)\bigr), \qquad \rho_n(\epsilon) := \epsilon^2 - N E_n
\]
where $E_n$ and $V_n$ are obtained from Theorem~\ref{thm:one-step-moments-uniform} by
replacing $P$ with $P^n$. For $n$ large enough that $\rho_n(\epsilon) > 0$, we have
\[
  \mathbb{P}_\pi\bigl(\diff^P_{2,\epsilon}(f_\pi \to \tau) > n\bigr)
  \le
  \frac{N^2 V_n}{N^2 V_n + \rho_n(\epsilon)^2}.
\]
\end{theorem}

\begin{proof}
Since $P$ is bistochastic with uniform stationary distribution $\tau = \frac{1}{N}\mathbf{1}$, for every probability vector $x$,
\[
  \|x - \tau\|_2^2
  = \sum_{i=1}^N \Bigl(x_i - \frac{1}{N}\Bigr)^2
  = N \cdot \frac{1}{N}\sum_{i=1}^N \Bigl(x_i - \frac{1}{N}\Bigr)^2
  = N\,\sqd(x).
\]
Applying this to $x = f_\pi P^n$,
\[
  \|f_\pi P^n - \tau\|_2^2 = N\,\sqd(f_\pi P^n).
\]
Therefore
\begin{align*}
  \diff^P_{2,\epsilon}(f_\pi \to \tau) > n
  &\iff \|f_\pi P^n - \tau\|_2 \ge \epsilon \\
  &\iff \|f_\pi P^n - \tau\|_2^2 \ge \epsilon^2 \\
  &\iff N\,\sqd(f_\pi P^n) \ge \epsilon^2.
\end{align*}

Since $P$ is normal, $P^n$ is also normal, and
\[
  P^n(P^n)^\top = U\,\diag(\mu_1^n, \dots, \mu_N^n)\,U^\top,
\]
where $PP^\top = U\,\diag(\mu_1,\dots,\mu_N)\,U^\top$ is an orthogonal diagonalization with $\mu_1 = 1$. In particular, $P^n$ is bistochastic, since $P^n\mathbf{1} = \mathbf{1}$ and $\mathbf{1}^\top P^n = \mathbf{1}^\top$. Therefore $M_n := P^n(P^n)^\top$ satisfies $M_n\mathbf{1} = \mathbf{1}$, and Theorem~\ref{thm:one-step-moments-uniform} applies with $P$ replaced by $P^n$.
This gives
\begin{align*}
  E_n
  &:= \mathbb{E}_\pi\bigl[\sqd(f_\pi P^n)\bigr]
  = \sqd(f)\,\frac{\tr(M_n) - 1}{N-1}, \\[4pt]
  V_n
  &:= \Var_\pi\bigl(\sqd(f_\pi P^n)\bigr)
  = \mathbb{E}_\pi\bigl[\sqd(f_\pi P^n)^2\bigr] - E_n^2,
\end{align*}
where $\mathbb{E}_\pi[\sqd(f_\pi P^n)^2]$ is given explicitly by Theorem~\ref{thm:one-step-moments-uniform} with $M$, $\Delta$ replaced by $M_n$, $\Delta_n$.

Since $\|f_\pi P^n - \tau\|_2^2 = N\,\sqd(f_\pi P^n)$, it follows by linearity and scaling that
\begin{align*}
  \mathbb{E}_\pi\bigl[\|f_\pi P^n - \tau\|_2^2\bigr]
  &= N E_n, \\[4pt]
  \Var_\pi\bigl(\|f_\pi P^n - \tau\|_2^2\bigr)
  &= N^2 V_n.
\end{align*}

\textbf{Cantelli's inequality.}
Let $X := \|f_\pi P^n - \tau\|_2^2$, so that $\mathbb{E}_\pi[X] = NE_n$
and $\Var_\pi(X) = N^2 V_n$. We wish to bound
$\mathbb{P}_\pi(X \ge \epsilon^2)$. Assuming $\rho_n(\epsilon) :=
\epsilon^2 - NE_n > 0$, we write
\[
  \mathbb{P}_\pi(X \ge \epsilon^2)
  = \mathbb{P}_\pi\bigl(X - NE_n \ge \epsilon^2 - NE_n\bigr)
  = \mathbb{P}_\pi\bigl(X - NE_n \ge \rho_n(\epsilon)\bigr).
\]
Cantelli's inequality states that for any random variable $X$ with finite variance and any $t > 0$,
\[
  \mathbb{P}(X - \mathbb{E}[X] \ge t)
  \le \frac{\Var(X)}{\Var(X) + t^2}.
\]
Applying this with $t = \rho_n(\epsilon) > 0$ and $\Var_\pi(X) = N^2 V_n$,
\[
  \mathbb{P}_\pi\bigl(X - NE_n \ge \rho_n(\epsilon)\bigr)
  \le \frac{N^2 V_n}{N^2 V_n + \rho_n(\epsilon)^2}.
\]
Since
\[
  \mathbb{P}_\pi\bigl(\diff^P_{2,\epsilon}(f_\pi \to \tau) > n\bigr)
  = \mathbb{P}_\pi(X \ge \epsilon^2),
\]
the stated bound follows.
\end{proof}

\subsection{Efficient bounds for permutation $p$-tests}

Bringing together Theorems~\ref{thm:one-step-moments-uniform} and~\ref{thm:hp-normal-permutation}, we obtain a strategy for bounding the $p$-value under a permutation null model when $P$ is a symmetric matrix (for example, the Metropolis-Hastings transition matrix). This strategy avoids having to sample many permutations. Given that the number of permutations grows like $N!$, this is very useful for large $N$. Suppose we are given $P$ and an observed probability distribution $f$. We perform the following steps.

\begin{enumerate}
    \item Compute $n = \diff_{2,\epsilon}^P(f \to \tau)$, the observed diffusion distance.
    \item Obtain eigenvalues $\lambda_k$ and eigenvectors $v_k$ for $P$. 
    \item $M_n := P^n (P^n)^\top = P^{2n}$ since $P$ is symmetric, and $P^{2n}$ has eigenvalues $\lambda_k^{2n}$ and the same eigenvectors $v_k$. Observation~\ref{obs:symmetricP} applied to $P^{n}$ therefore gives the three structural quantities directly from the
      spectrum of $P$:
      \[
        \tr(M_n) = \sum_k \lambda_k^{2n},
        \qquad
        \tr(M_n^2) = \sum_k \lambda_k^{4n},
        \qquad
        \Delta_n = \sum_i \Big( \sum_k \lambda_k^{2n} v_{ik}^2 \Big)^{\!2}.
      \]
      Note that the eigendecomposition of $P$ is computed once for a given graph; each subsequent computation then requires only raising the eigenvalues to a power. Substituting these into Theorem~\ref{thm:one-step-moments-uniform} yields
      $E_n := \mathbb{E}_\pi\bigl[\sqd(f_\pi P^n)\bigr]$ and
      $V_n := \Var_\pi\bigl(\sqd(f_\pi P^n)\bigr)$.
    \item Provided $\epsilon^2 - N \mathbb{E}_\pi\bigl[\sqd(f_\pi P^n)\bigr] > 0$, we use Theorem~\ref{thm:hp-normal-permutation} to obtain an upper bound for the required $p$-value $\mathbb{P}_\pi\bigl(\diff^P_{2,\epsilon}(f_\pi \to \tau) > n\bigr)$. If the inequality is violated, then the significance is too marginal to be determined by moment methods, and direct sampling is required.
\end{enumerate}

For very large $N$, we can compute only the $k$ largest eigenvalues and their eigenvectors, and treat Step 3 as an approximation. The quality of the approximation depends on the decay of the eigenvalues of $P$, which one might hope to be fast for geospatial graphs of fine granularity.

\section{Numerical experiments}
\label{sec:experiments}

\subsection{Measuring power on stochastic block models}\label{sec:power}

We test the power of the diffusion distance permutation test on randomly generated graphs and distributions. We use a stochastic block model (SBM) to generate a graph $G$ with $40$ nodes divided evenly into two classes $A$ and $B$. An edge appears between nodes of the same class with probability $0.5$ and between nodes of different classes with probability $0.05$. For each node, a random scalar value is drawn uniformly from $[0,1]$. Values for nodes in class $B$ are uniformly shifted by a parameter $\alpha \geq 0$; see Figure~\ref{fig:sampledataset}. All values are normalized to sum to $1$ before computing diffusion distance. As $\alpha$ increases, the spatial dependence in the data increases and should be detectable by diffusion distance and Moran's $I$ tests. We run $1,000$ trials with $1,000$ permutations each, and reject the null hypothesis of spatial independence for \textit{p}-values less than $0.05$. We use $p=2$ and $\epsilon = 0.01$ throughout for diffusion distance based on the heuristics in Remark~\ref{rem:heuristic}. Figure~\ref{fig:rejection} shows that when $\alpha = 0$, there is a false positive rate of about $0.05$ for both methods as expected. As $\alpha$ increases, diffusion distance outperforms Moran's $I$ in terms of power, likely because the data is locally noisy but globally biased. As an additional precaution, we rerun the experiment with $\epsilon = 0.001$ and get almost identical results (see the Appendix).

\begin{figure}[h]
    \centering
    \begin{subfigure}{0.58\textwidth}
\includegraphics[width=\linewidth]{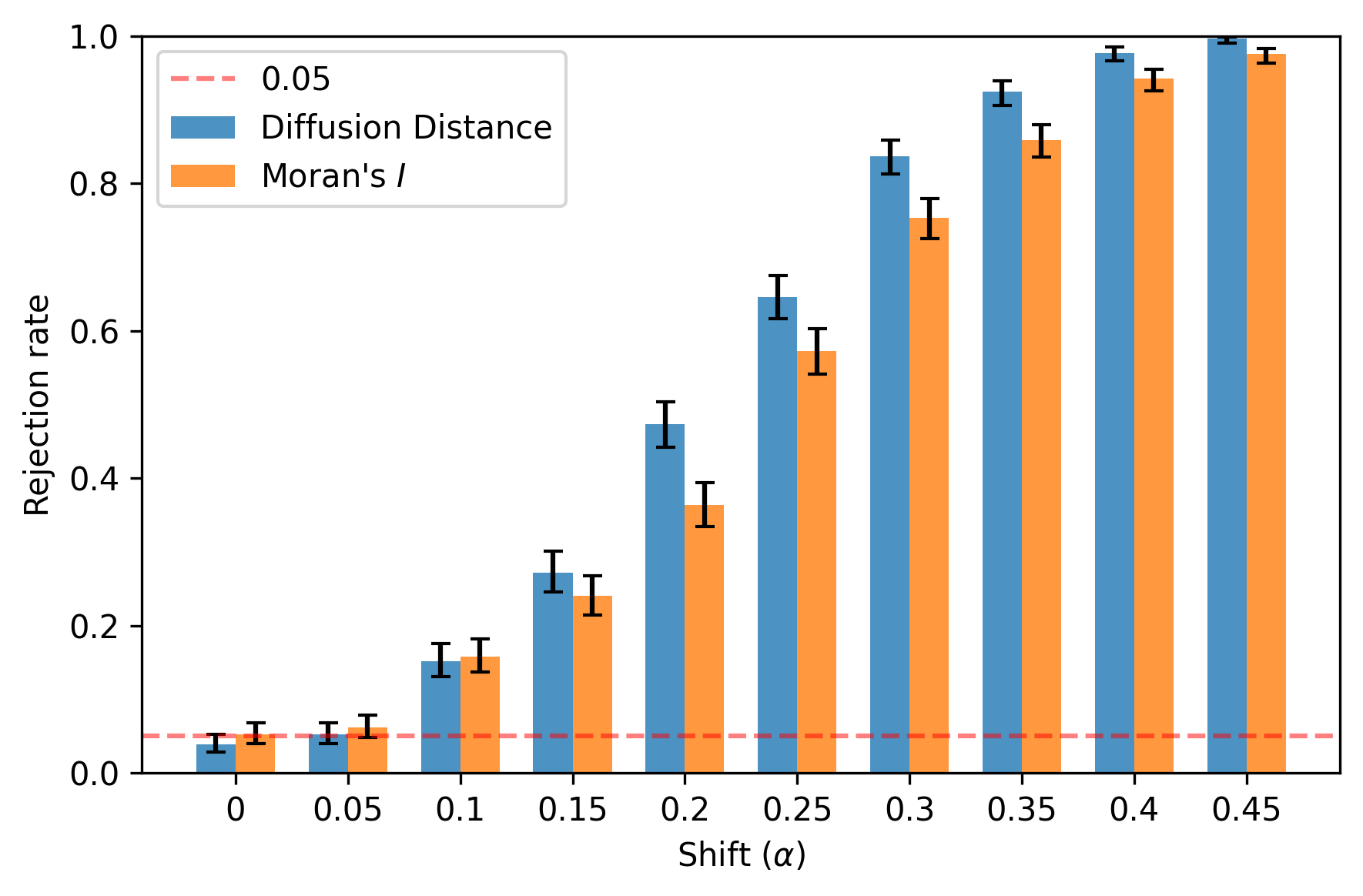}
\caption{Rejection rates with 95\% Wilson confidence interval}\label{fig:rejection}
    \end{subfigure}
    \begin{subfigure}{0.4\textwidth}
\includegraphics[width=\linewidth]{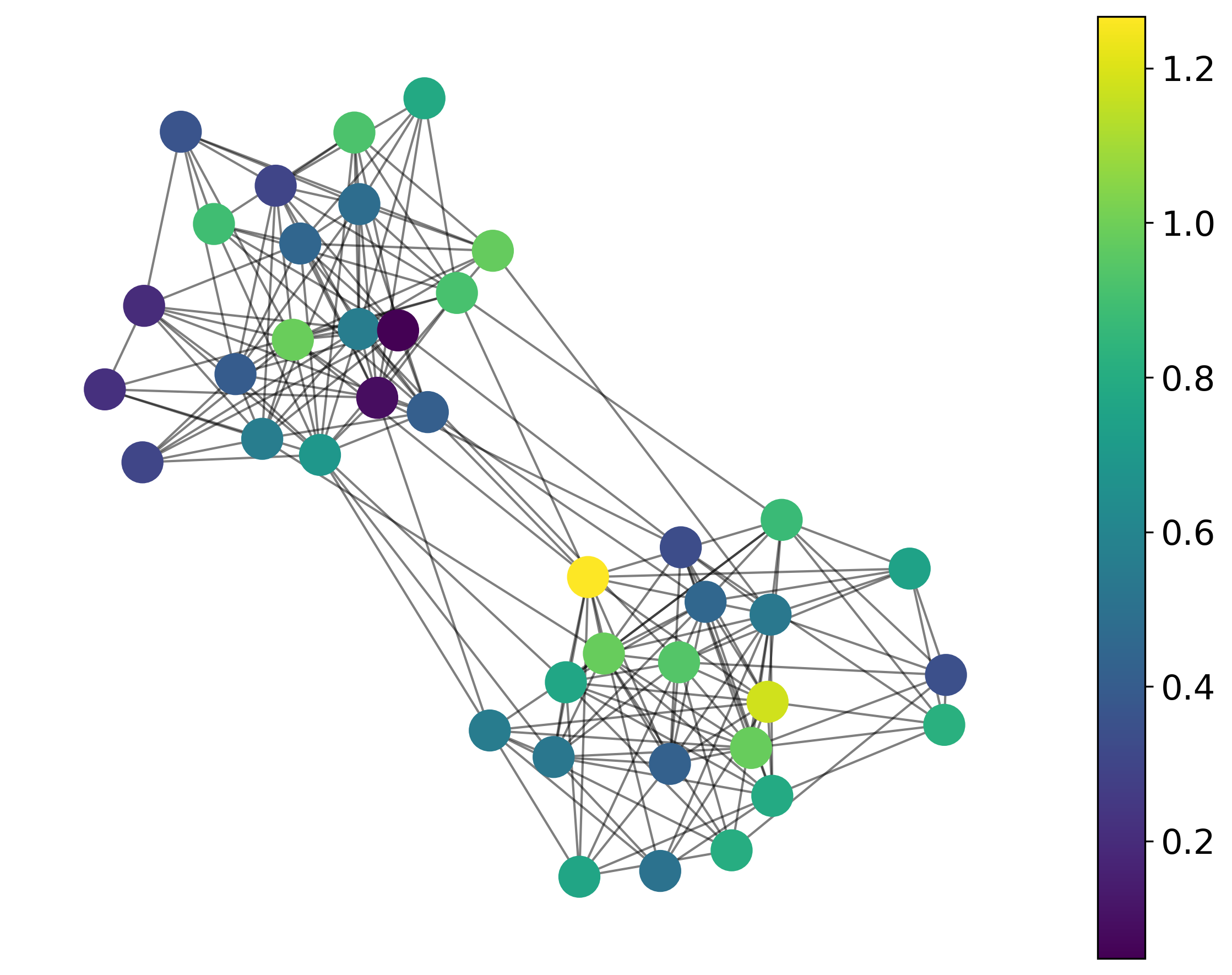}
\caption{Sample dataset with $\alpha = 0.2$}\label{fig:sampledataset}
    \end{subfigure}
    \caption{Power analysis of diffusion distance and Moran's $I$ permutation tests on SBM graphs.}
    \label{fig:power}
\end{figure}

\subsection{U.S. city demographics}\label{sec:demos}

We study geospatial clustering using a dataset of $100$ U.S. cities with boundaries from the CDC's $500$ city project~\cite{CDC500CitiesBoundaries}, and Census demographic data obtained from NHGIS~\cite{IPUMS2021}. For each city, we use the dual graph of the tract map, with vertices representing census tracts and edges connecting geographically adjacent tracts; all computations are restricted to the largest connected component. Figure~\ref{fig:chicago_share_dualgraph} illustrates the dual graph  construction for Chicago. For tract $i$, let $B_i$ and $T_i$ denote the Black and total populations, respectively, and set $s_i = B_i/T_i$, with $s_i = 0$ when
$T_i = 0$. Our initial distribution $f$ is defined by $f_i = s_i / \sum_i s_i$. For each city, we compute the Metropolis Hastings diffusion distance to uniformity using $f$. We choose $\epsilon = 0.25$ and  $p = 1$ based on Remark~\ref{rem:heuristic}. We then compare diffusion distance with Moran's $I$, computed on the same census-tract graph for each city. Figure~\ref{fig:mi_vs_diffusion} shows a clear positive association between the two measures when diffusion distance is plotted on a log-scale (as expected, since it measures a number of iterations of roughly exponential decay), with a rank correlation of $0.728$ over 100 cities. Diffusion distance and Moran's $I$ therefore capture overlapping but nonidentical aspects of spatial structure.  

To more closely observe the types of urban segregation treated differently by the two methods, we focus on cities with 20\% or higher Black population percentage citywide. Among these, the six cities labeled in Figure~\ref{fig:mi_vs_diffusion} stand out as being ranked significantly higher by Moran's $I$ than by diffusion distance. Roughly speaking, this means they are more locally clustered than they are globally clustered. Looking at the city graphs in Figure~\ref{fig:citygraph6}, we see at least one common pattern: non-Black areas near the city center totally or partially surrounded by high Black population areas (this pattern is much less pronounced in Toledo than in the other five, however). This could be the trace left by actions such as ``slum clearance'' in the 20$^{\mathrm{th}}$ century~\cite{LAVOICE2024105153} which displaced Black residents from specific areas. The low diffusion distance rank indicates that while these cities are segregated, this segregation is not stable under natural graph dynamics. The boundary between Black and non-Black neighborhoods is not a natural bottleneck and could be an artifact of specific historical events.

\begin{figure}[h]
  \centering
\includegraphics[width=0.35\textwidth]{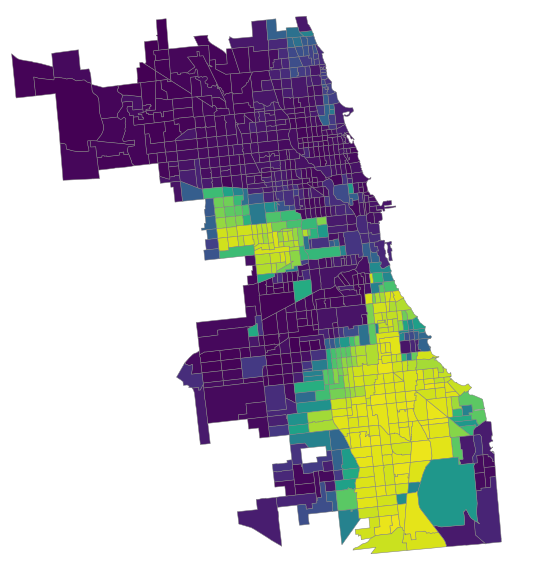}%
  \includegraphics[width=0.45\textwidth]{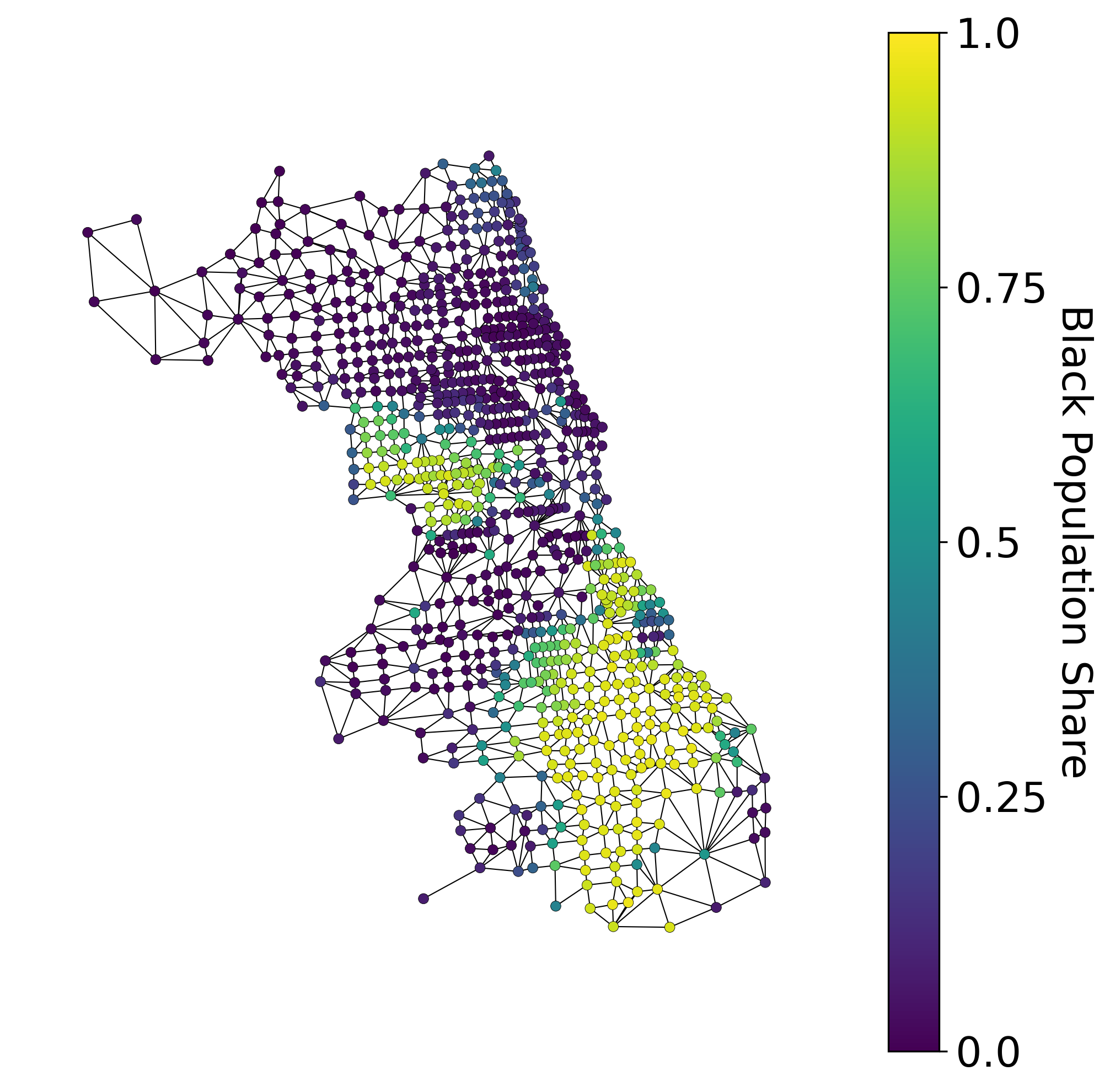}
  \caption{Chicago tract map and corresponding dual graph, both colored
    by tract-level Black population share.}
  \label{fig:chicago_share_dualgraph}
\end{figure}

\begin{figure}[h]
  \centering
  \begin{subfigure}{0.49\textwidth}
        \includegraphics[width=\textwidth]{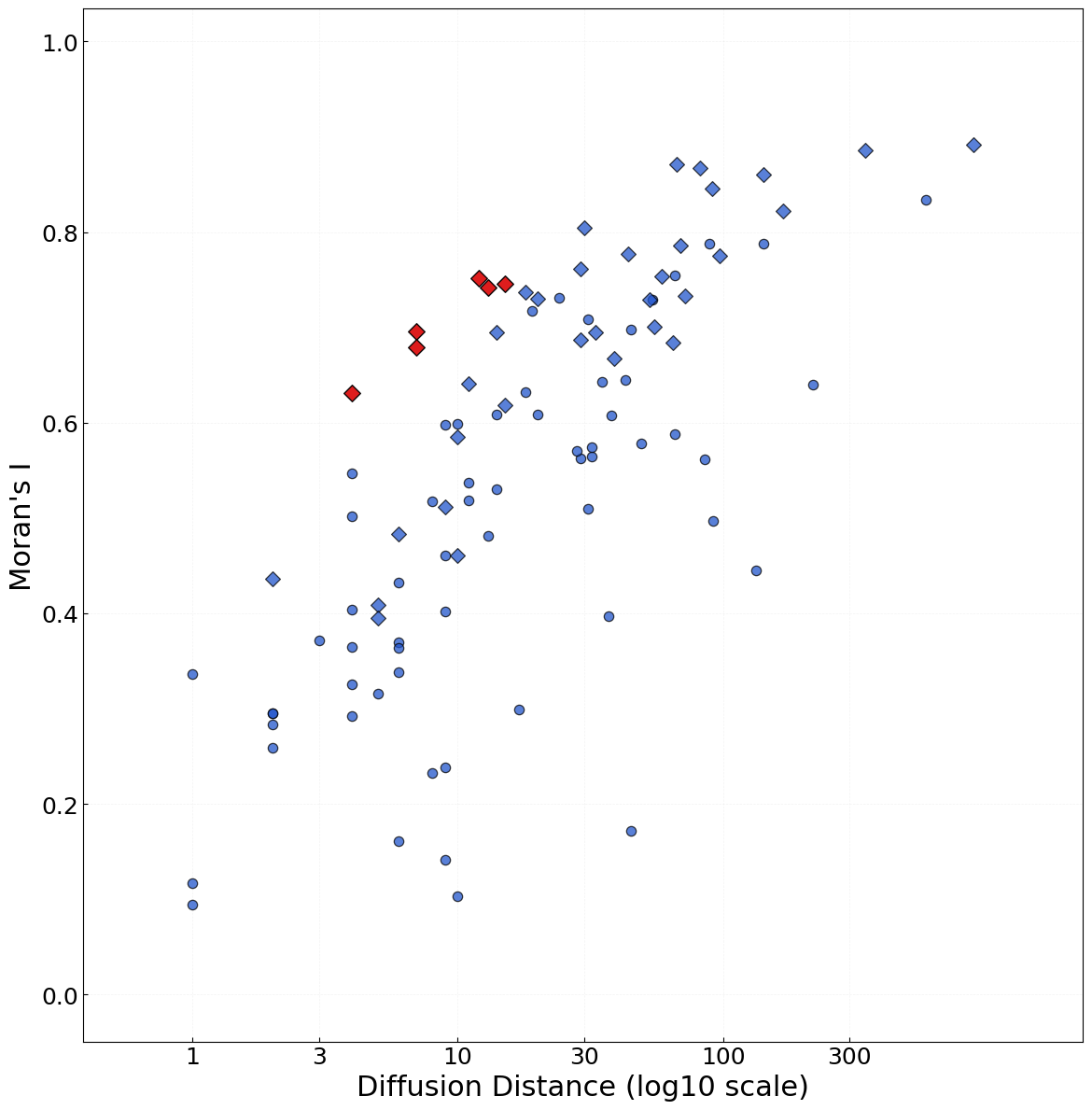}
  \end{subfigure}\hfill
   \begin{subfigure}{0.49\textwidth}
         \includegraphics[width=\textwidth]{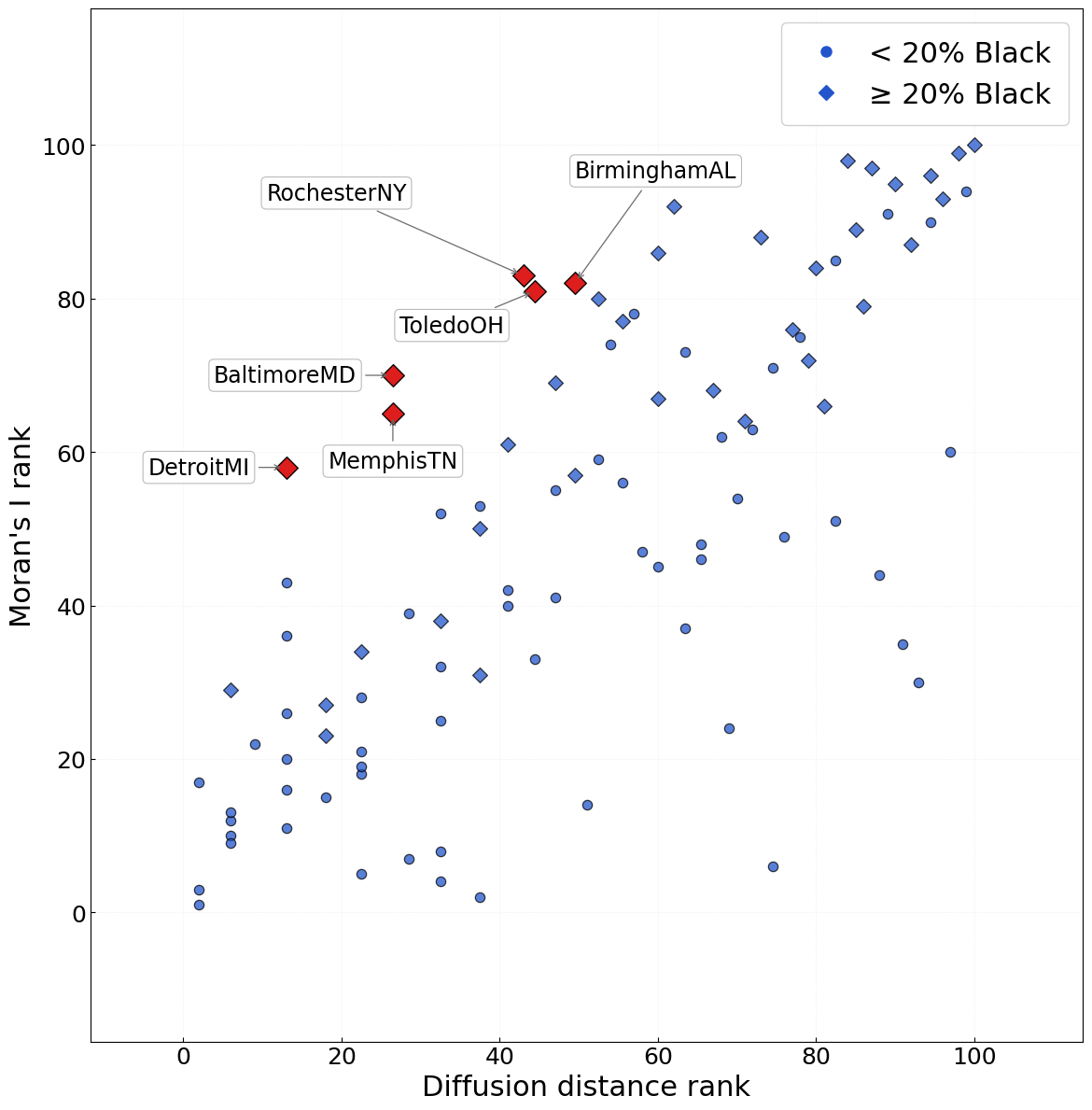}
  \end{subfigure}
      \caption{Comparing Metropolis-Hastings diffusion distance to uniformity and Moran's $I$ on Black population data for 100 U.S. cities using scatter plots for the raw values with a log $x$-axis (left), and for the ranks among the 100 cities (right). }  \label{fig:mi_vs_diffusion}

\end{figure}

\begin{figure}[h]
\centering
\begin{subfigure}{0.28\textwidth}
    \centering
    \includegraphics[width=\textwidth]{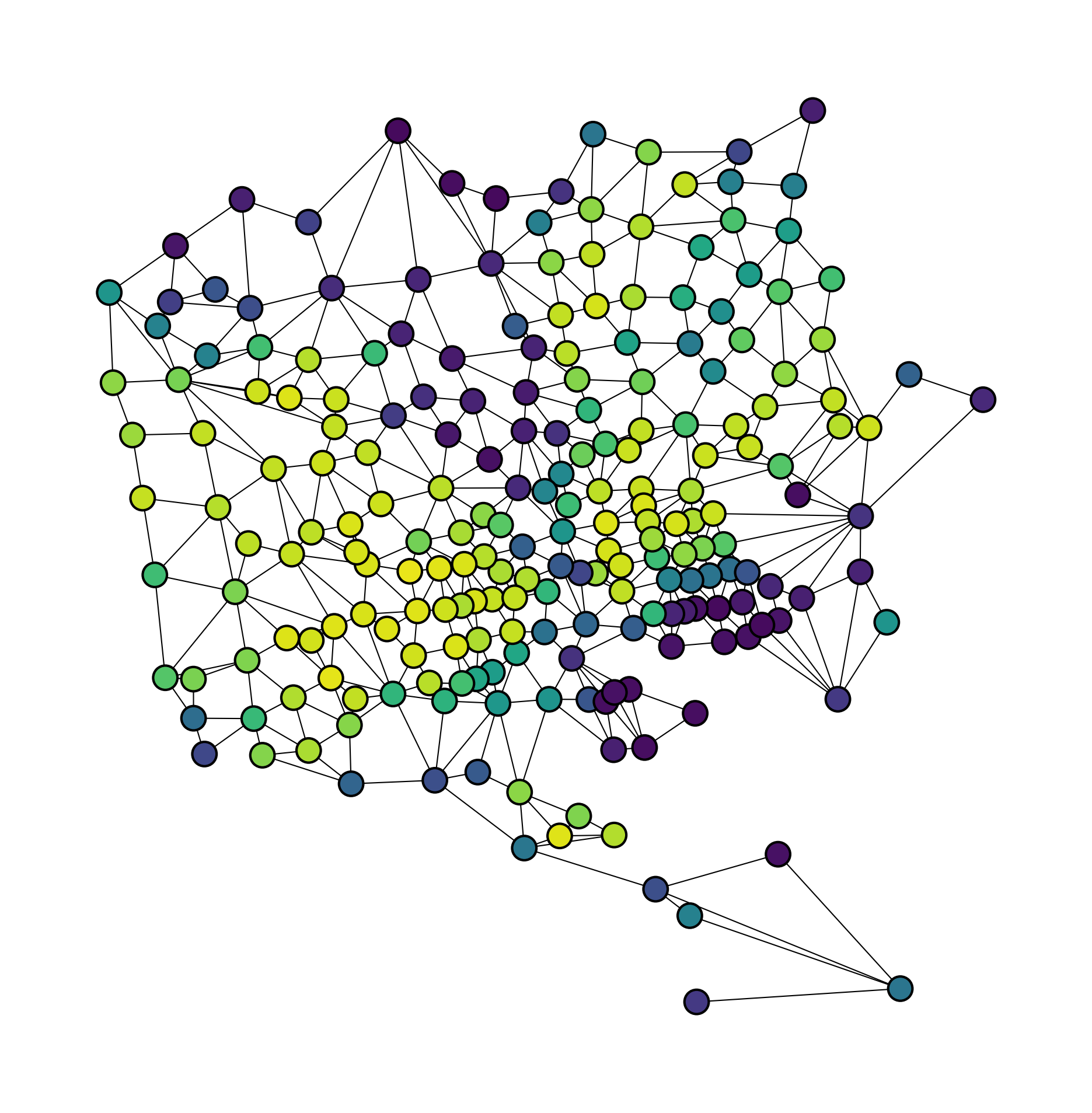}
    \caption{Baltimore MD}
\end{subfigure}
\begin{subfigure}{0.28\textwidth}
    \centering
    \includegraphics[width=\textwidth]{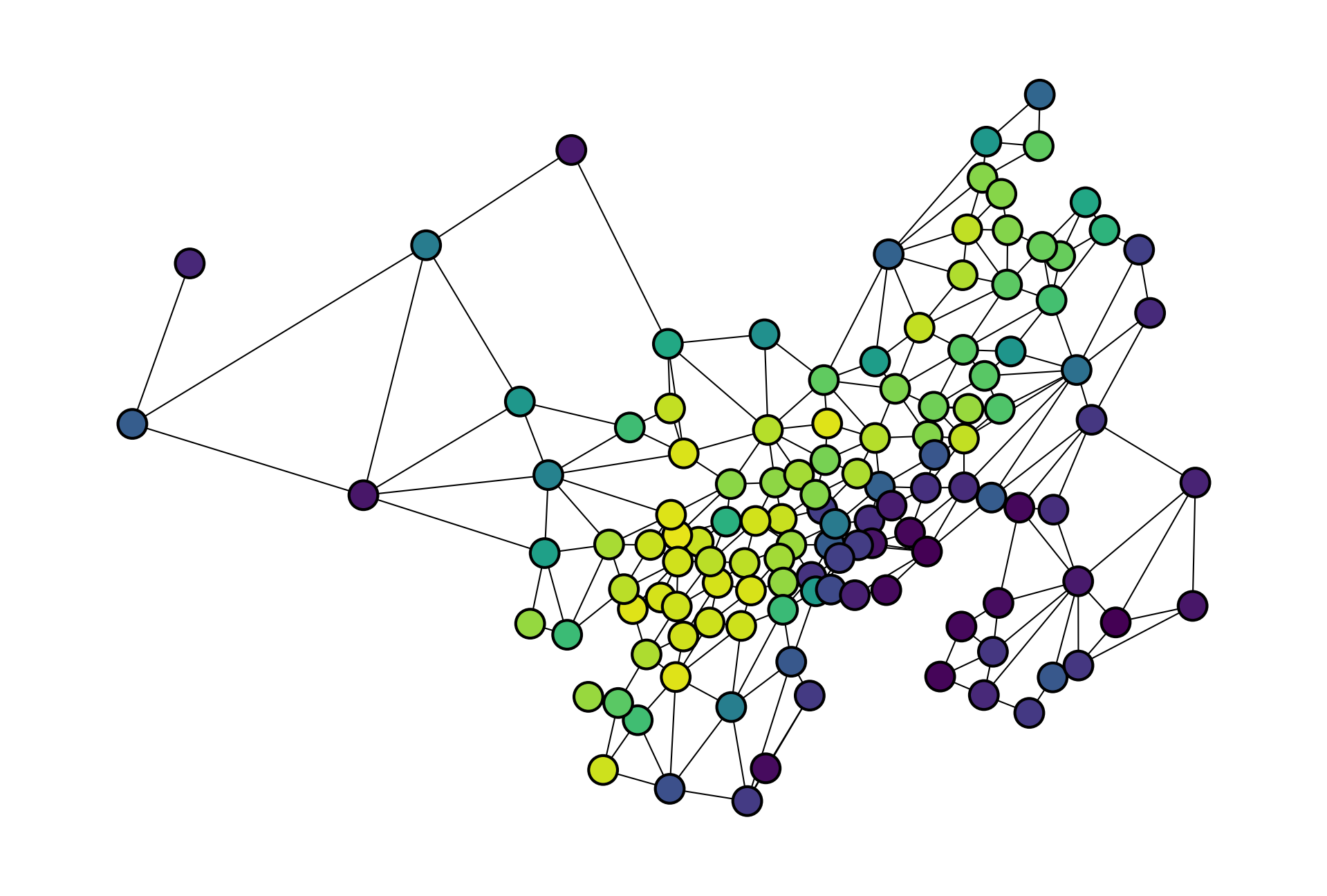}
    \caption{Birmingham AL}
\end{subfigure}
\begin{subfigure}{0.28\textwidth}
    \centering
    \includegraphics[width=\textwidth]{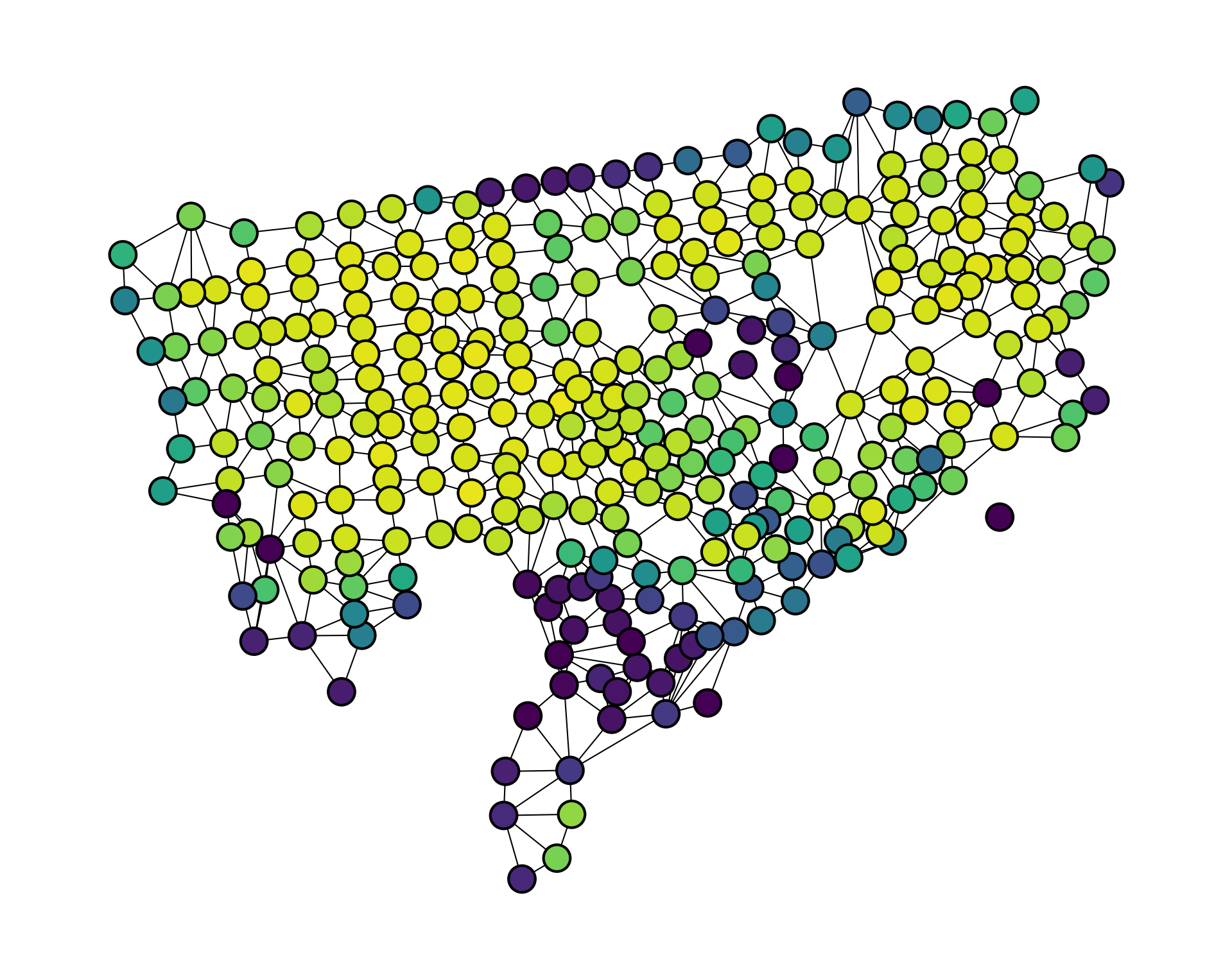}
    \caption{Detroit MI}
\end{subfigure}

\begin{subfigure}{0.28\textwidth}
    \centering
    \includegraphics[width=\textwidth]{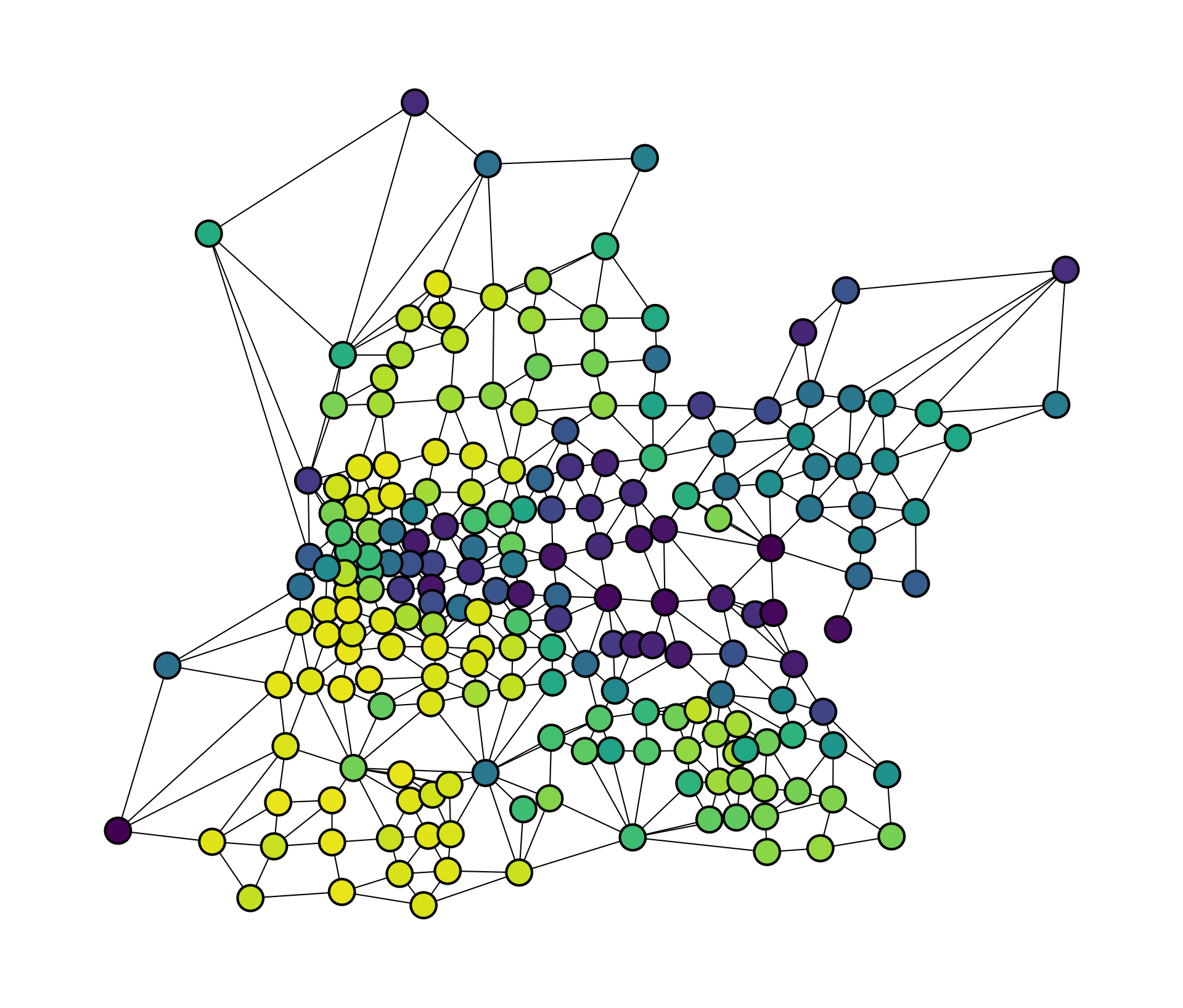}
    \caption{Memphis TN}
\end{subfigure}
\begin{subfigure}{0.28\textwidth}
    \centering
    \includegraphics[width=\textwidth]{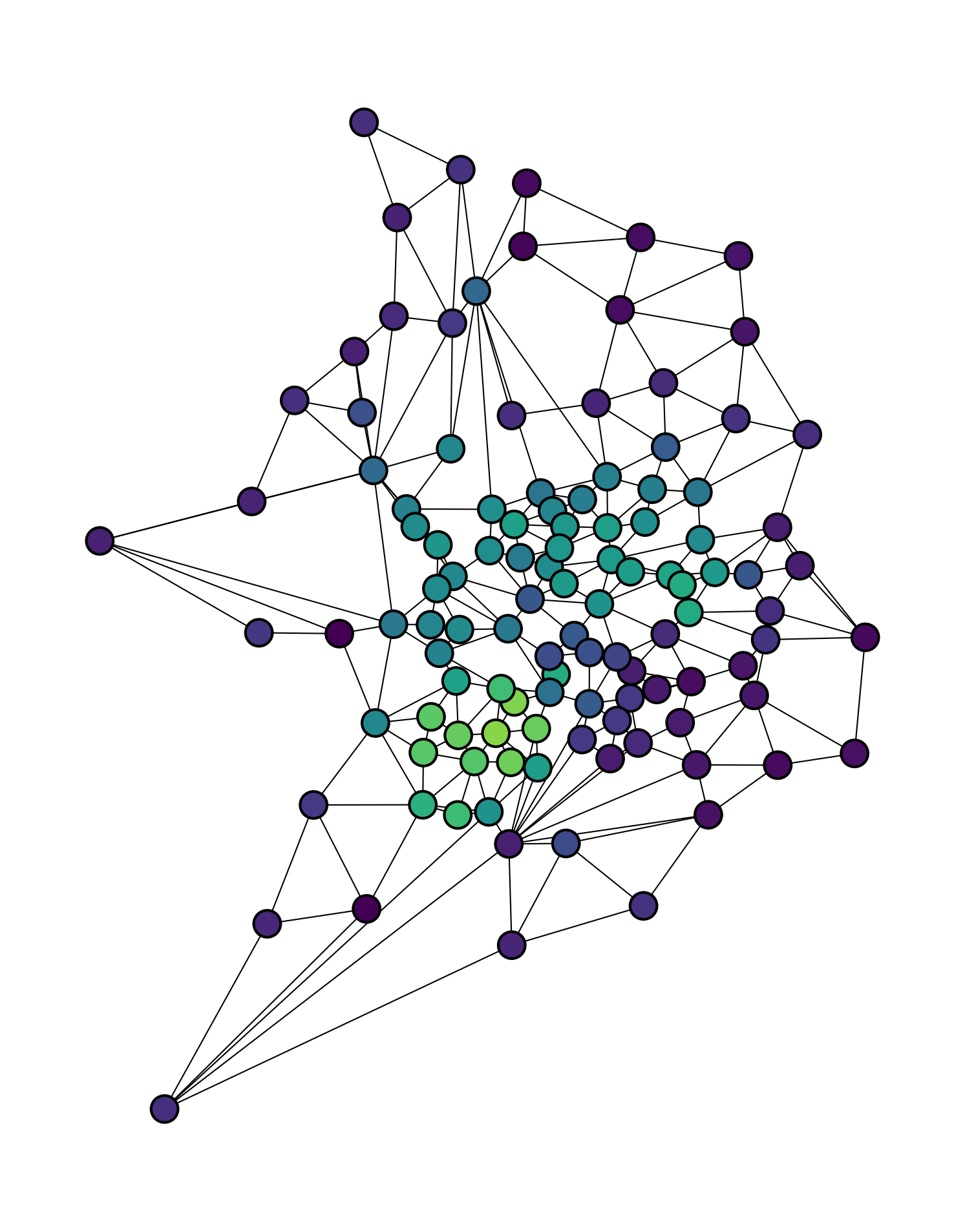}
    \caption{Rochester NY}
\end{subfigure}
\begin{subfigure}{0.4\textwidth}
    \centering
    \includegraphics[width=\textwidth]{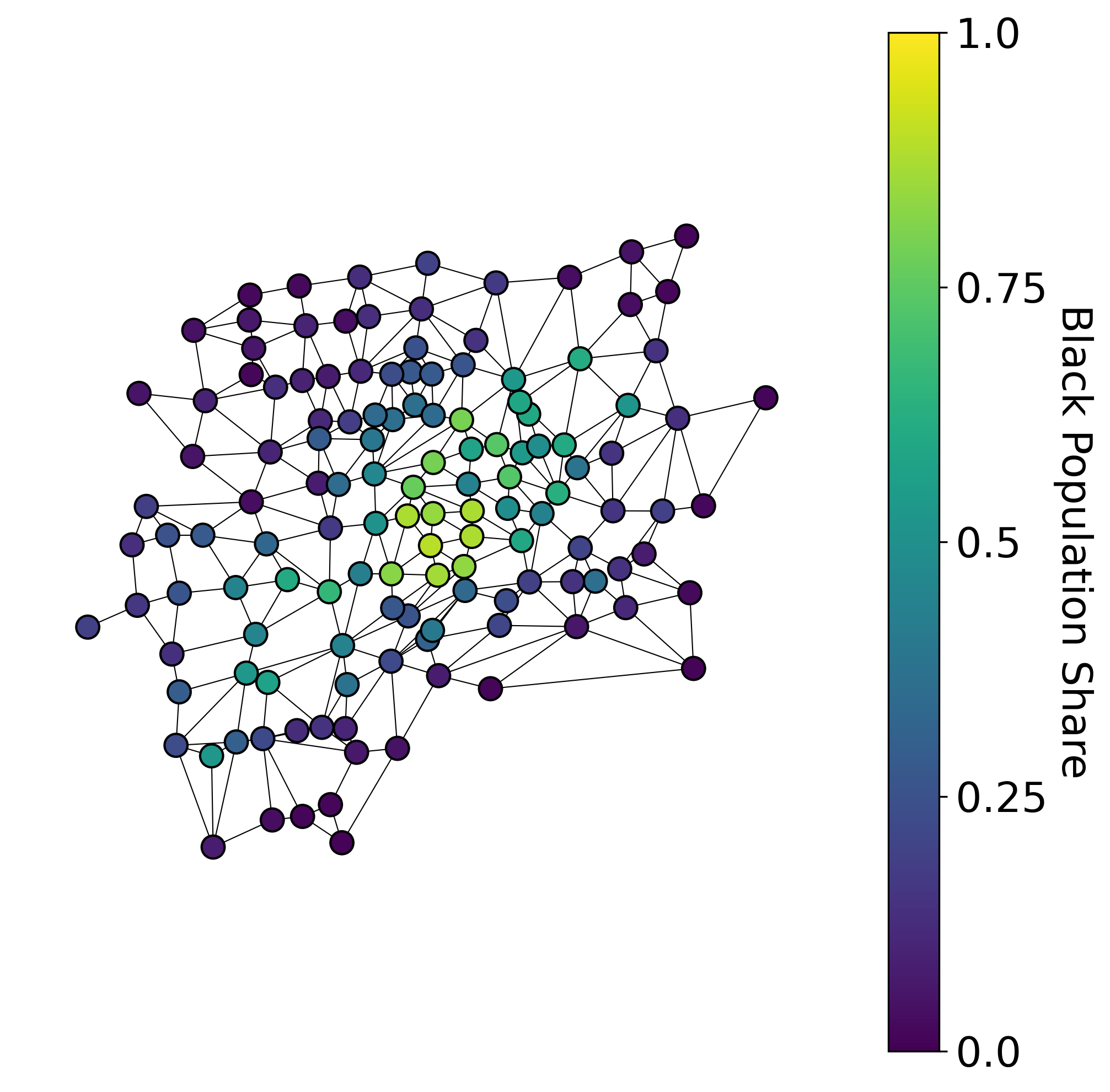}
    \caption{Toledo OH}
\end{subfigure}
\caption{Dual graphs of six cities with substantial Black populations whose rank by Moran's $I$ (local clustering) significantly exceeds their rank by diffusion distance (global clustering). See the red points in Figure~\ref{fig:mi_vs_diffusion}. Color scale shows the Black population share by Census tract.}
\label{fig:citygraph6}
\end{figure}

\clearpage 

\bibliographystyle{plain}
\bibliography{refs}
\appendix

\section{Computational and data details}

The experiments in Section~\ref{sec:experiments} were run on laptops/MacBooks using Python. The largest city dual graph is New York NY with 1848 nodes. Computing the Moran's $I$ and diffusion distance values in Figure~\ref{fig:mi_vs_diffusion} is almost instantaneous since the sparse Metropolis-Hastings transition matrix is easy to compute and exponentiate as needed. The power analysis in Figure~\ref{fig:power} requires multiple permutations and multiple trials and so takes longer. It takes just under a second to perform a single diffusion distance permutation test on one SBM graph with 1,000 permutations; this was on a MacBook with a 2.3 GHz 8-Core Intel Core i9 processor and 64 GB of RAM. The entire experiment for Figure~\ref{fig:power} therefore takes about 10 minutes per $\alpha$ value. For much larger graphs, we would expect to use the approximations in Section~\ref{sec:statistics}. A version of Figure~\ref{fig:power} with $\epsilon = 0.001$ is shown in Figure~\ref{fig:power2}, with almost identical results.

\begin{figure}[h]
    \centering
\includegraphics[width=0.5\linewidth]{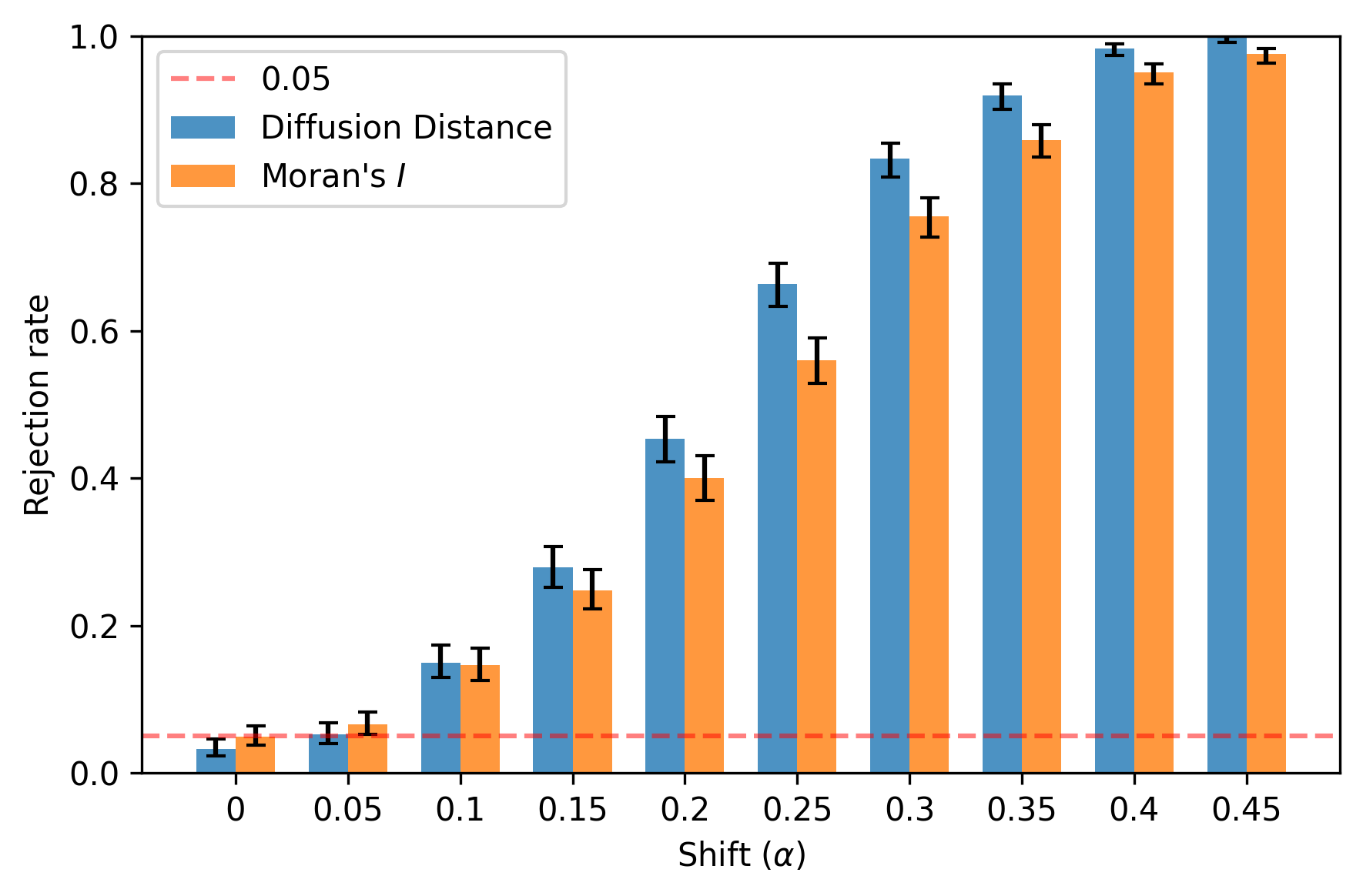}
    \caption{Rejection rates for $\epsilon = 0.001$ with the same setup as in Section~\ref{sec:power}.}
    \label{fig:power2}
\end{figure}

Census data obtained from NHGIS is used in accordance with their licence; see \url{https://www.nhgis.org/citation-and-use-nhgis-data}. We use the preprocessed city graphs published by the authors of \cite{kauba2024topological} and used in their analysis in that paper.

\section{Proof of Theorem~\ref{thm:one-step-moments-uniform}}
\label{proof:thm:one-step-moments-uniform}

\textbf{Preliminaries.}
Since $f$ is a probability vector, $\sum_{i=1}^N f_i = 1$, and therefore
\[
  \sum_{i=1}^N Z_i
  = \sum_{i=1}^N\!\Bigl(f_i - \tfrac{1}{N}\Bigr)
  = 1 - 1
  = 0.
\]
Since $P$ is bistochastic, $M := PP^\top$ is symmetric and satisfies
$M\mathbf{1} = PP^\top\mathbf{1} = P\mathbf{1} = \mathbf{1}$, so
\begin{align*}
  \sum_{i,j=1}^N M_{ij}
  &= \mathbf{1}^\top M\mathbf{1}
   = \mathbf{1}^\top\mathbf{1}
   = N, \\[4pt]
  \sum_{i,j,k,\ell=1}^N M_{ij}M_{k\ell}
  &= \Bigl(\sum_{i,j=1}^N M_{ij}\Bigr)^2
   = N^2.
\end{align*}
Also, since $M$ is symmetric, $\tr(M^2) = \sum_{i,j=1}^N M_{ij}^2$.
Because $\tau = \frac{1}{N}\mathbf{1}$ is stationary and $P\mathbf{1} = \mathbf{1}$,
\[
  f_\pi P - \tau
  = \bigl(f_\pi - \tau\bigr)P
  = Z_\pi P,
\]
and therefore
\begin{align*}
  \sqd(f_\pi P)
  &= \frac{1}{N}\|Z_\pi P\|_2^2
   = \frac{1}{N}\,Z_\pi PP^\top Z_\pi^\top \\
  &= \frac{1}{N}\,Z_\pi M Z_\pi^\top
   = \frac{1}{N}\sum_{i,j=1}^N M_{ij}\,Z_{\pi(i)}Z_{\pi(j)}.
\end{align*}

\textbf{Permutation moments.}
Because $\pi$ is uniformly distributed over all permutations of $\{1,\dots,N\}$, for any $m$ distinct positions $i_1,\dots,i_m$, the vector $(\pi(i_1),\dots,\pi(i_m))$ is uniformly distributed over all ordered $m$-tuples of distinct elements of $\{1,\dots,N\}$. We compute the following eight moments.

\medskip
\noindent\textit{(a)} Since $\pi(i)$ is uniform on $\{1,\dots,N\}$,
\[
  \mathbb{E}_\pi[Z_{\pi(i)}]
  = \frac{1}{N}\sum_{a=1}^N Z_a
  = 0.
\]

\noindent\textit{(b)} Similarly,
\[
  \mathbb{E}_\pi[Z_{\pi(i)}^2]
  = \frac{1}{N}\sum_{a=1}^N Z_a^2
  = \sqd(f).
\]

\noindent\textit{(c)} For $i \ne j$,
\[
  \mathbb{E}_\pi[Z_{\pi(i)}Z_{\pi(j)}]
  = \frac{1}{N(N-1)}\sum_{a \ne b} Z_a Z_b.
\]
Now
\[
  \sum_{a \ne b} Z_a Z_b
  = \Bigl(\sum_{a=1}^N Z_a\Bigr)^2 - \sum_{a=1}^N Z_a^2
  = 0 - N\sqd(f)
  = -N\sqd(f),
\]
so
\[
  \mathbb{E}_\pi[Z_{\pi(i)}Z_{\pi(j)}]
  = -\frac{\sqd(f)}{N-1}.
\]

\noindent\textit{(d)}
\[
  \mathbb{E}_\pi[Z_{\pi(i)}^4]
  = \frac{1}{N}\sum_{a=1}^N Z_a^4
  = \frac{S_4(Z)}{N}.
\]

\noindent\textit{(e)} For $i \ne j$,
\[
  \mathbb{E}_\pi[Z_{\pi(i)}^3 Z_{\pi(j)}]
  = \frac{1}{N(N-1)}\sum_{a \ne b} Z_a^3 Z_b.
\]
Using $\sum_{b=1}^N Z_b = 0$, so that $\sum_{b \ne a} Z_b = -Z_a$,
\[
  \sum_{a \ne b} Z_a^3 Z_b
  = \sum_{a=1}^N Z_a^3 \sum_{b \ne a} Z_b
  = \sum_{a=1}^N Z_a^3(-Z_a)
  = -S_4(Z),
\]
and therefore
\[
  \mathbb{E}_\pi[Z_{\pi(i)}^3 Z_{\pi(j)}]
  = -\frac{S_4(Z)}{N(N-1)}.
\]

\noindent\textit{(f)} For $i \ne j$,
\[
  \mathbb{E}_\pi[Z_{\pi(i)}^2 Z_{\pi(j)}^2]
  = \frac{1}{N(N-1)}\sum_{a \ne b} Z_a^2 Z_b^2.
\]
Since $\sum_{a=1}^N Z_a^2 = N\sqd(f)$,
\[
  \sum_{a \ne b} Z_a^2 Z_b^2
  = \Bigl(\sum_{a=1}^N Z_a^2\Bigr)^2 - \sum_{a=1}^N Z_a^4
  = N^2\sqd(f)^2 - S_4(Z),
\]
so
\[
  \mathbb{E}_\pi[Z_{\pi(i)}^2 Z_{\pi(j)}^2]
  = \frac{N^2\sqd(f)^2 - S_4(Z)}{N(N-1)}.
\]

\noindent\textit{(g)} For pairwise distinct $i, j, k$,
\[
  \mathbb{E}_\pi[Z_{\pi(i)}^2 Z_{\pi(j)} Z_{\pi(k)}]
  = \frac{1}{N(N-1)(N-2)}
    \sum_{\substack{a,b,c \\ \text{distinct}}} Z_a^2 Z_b Z_c.
\]
For fixed $a$, since $\sum_{b \ne a} Z_b = -Z_a$,
\begin{align*}
  \sum_{\substack{b,c \ne a \\ b \ne c}} Z_b Z_c
  &= \Bigl(\sum_{b \ne a} Z_b\Bigr)^2 - \sum_{b \ne a} Z_b^2 \\
  &= Z_a^2 - \bigl(N\sqd(f) - Z_a^2\bigr) \\
  &= 2Z_a^2 - N\sqd(f).
\end{align*}
Therefore
\begin{align*}
  \sum_{\substack{a,b,c \\ \text{distinct}}} Z_a^2 Z_b Z_c
  &= \sum_{a=1}^N Z_a^2\bigl(2Z_a^2 - N\sqd(f)\bigr) \\
  &= 2S_4(Z) - N^2\sqd(f)^2,
\end{align*}
so
\[
  \mathbb{E}_\pi[Z_{\pi(i)}^2 Z_{\pi(j)} Z_{\pi(k)}]
  = \frac{2S_4(Z) - N^2\sqd(f)^2}{N(N-1)(N-2)}.
\]

\noindent\textit{(h)} For pairwise distinct $i, j, k, \ell$,
\[
  \mathbb{E}_\pi[Z_{\pi(i)} Z_{\pi(j)} Z_{\pi(k)} Z_{\pi(\ell)}]
  = \frac{1}{N(N-1)(N-2)(N-3)}
    \sum_{\substack{a,b,c,d \\ \text{distinct}}} Z_a Z_b Z_c Z_d.
\]
Expanding $0 = \bigl(\sum_{a=1}^N Z_a\bigr)^4$ and grouping by multiplicity
pattern,
\[
  0
  = S_4(Z)
  + 4\!\sum_{a \ne b} Z_a^3 Z_b
  + 6\!\sum_{a \ne b} Z_a^2 Z_b^2
  + 12\!\sum_{\substack{a,b,c \\ \text{distinct}}} Z_a^2 Z_b Z_c
  + \sum_{\substack{a,b,c,d \\ \text{distinct}}} Z_a Z_b Z_c Z_d.
\]
Substituting the sums from (e), (f), and (g),
\begin{align*}
  \sum_{\substack{a,b,c,d \\ \text{distinct}}} Z_a Z_b Z_c Z_d
  &= -S_4(Z)
   + 4S_4(Z)
   - 6\bigl(N^2\sqd(f)^2 - S_4(Z)\bigr)
   - 12\bigl(2S_4(Z) - N^2\sqd(f)^2\bigr) \\
  &= 3N^2\sqd(f)^2 - 6S_4(Z),
\end{align*}
and therefore
\[
  \mathbb{E}_\pi[Z_{\pi(i)} Z_{\pi(j)} Z_{\pi(k)} Z_{\pi(\ell)}]
  = \frac{3N^2\sqd(f)^2 - 6S_4(Z)}{N(N-1)(N-2)(N-3)}.
\]

\textbf{First moment.}
Since $\mathbb{E}_\pi[Z_{\pi(i)}] = 0$ for all $i$, the standard identity
for expected quadratic forms gives
\[
  \mathbb{E}_\pi[Z_\pi M Z_\pi^\top]
  = \operatorname{tr}\!\bigl(M\,\operatorname{Cov}_\pi(Z_\pi)\bigr).
\]
From moments (b) and (c),
\[
  \bigl(\operatorname{Cov}_\pi(Z_\pi)\bigr)_{ij}
  =
  \begin{cases}
    \sqd(f) & i = j, \\[4pt]
    -\dfrac{\sqd(f)}{N-1} & i \ne j,
  \end{cases}
\]
which can be written as
\[
  \operatorname{Cov}_\pi(Z_\pi)
  = \frac{N\sqd(f)}{N-1}\!\left(I - \frac{1}{N}J\right),
\]
where $J = \mathbf{1}\mathbf{1}^\top$. Therefore
\begin{align*}
  \mathbb{E}_\pi[\sqd(f_\pi P)]
  &= \frac{1}{N}\operatorname{tr}\!\bigl(M\,\operatorname{Cov}_\pi(Z_\pi)\bigr) \\
  &= \frac{\sqd(f)}{N-1}
     \operatorname{tr}\!\!\left(M\!\left(I - \frac{1}{N}J\right)\right) \\
  &= \frac{\sqd(f)}{N-1}
     \!\left(\tr(M) - \frac{\tr(MJ)}{N}\right).
\end{align*}
Since $M\mathbf{1} = \mathbf{1}$, we have $\tr(MJ) = \mathbf{1}^\top M\mathbf{1} = N$, and so
\[
  \mathbb{E}_\pi\bigl[\sqd(f_\pi P)\bigr]
  = \sqd(f)\,\frac{\tr(M) - 1}{N-1}.
\]

\textbf{Second moment.}
Squaring the expression for $\sqd(f_\pi P)$ and taking expectations,
\begin{equation}
  N^2\,\mathbb{E}_\pi\bigl[\sqd(f_\pi P)^2\bigr]
  = \sum_{i,j,k,\ell=1}^N M_{ij} M_{k\ell}\,
    \mathbb{E}_\pi\!\bigl[Z_{\pi(i)} Z_{\pi(j)} Z_{\pi(k)} Z_{\pi(\ell)}\bigr].
  \label{eq:second_moment_master}
\end{equation}
The expected value on the right depends only on the multiplicity pattern
of $(i,j,k,\ell)$. We partition $[N]^4$ into five sets:
\begin{align*}
  \Gamma_4     &:= \{(i,j,k,\ell) : i=j=k=\ell\}, \\
  \Gamma_{31}  &:= \{(i,j,k,\ell) : \text{exactly three indices are equal}\}, \\
  \Gamma_{22}  &:= \{(i,j,k,\ell) : \text{indices form two distinct equal pairs}\}, \\
  \Gamma_{211} &:= \{(i,j,k,\ell) : \text{exactly one equal pair, other two distinct}\}, \\
  \Gamma_{1111}&:= \{(i,j,k,\ell) : \text{all four indices pairwise distinct}\},
\end{align*}
and write $\mathcal{T}_\alpha := \sum_{(i,j,k,\ell) \in \Gamma_\alpha} M_{ij}
M_{k\ell}$. Since these five sets partition $[N]^4$,
\[
  \mathcal{T}_4 + \mathcal{T}_{31} + \mathcal{T}_{22}
  + \mathcal{T}_{211} + \mathcal{T}_{1111} = N^2.
\]

\medskip
\noindent\textit{Computing $\mathcal{T}_4$.}
\[
  \mathcal{T}_4
  = \sum_{a=1}^N M_{aa}^2
  = \Delta.
\]

\noindent\textit{Computing $\mathcal{T}_{31}$.}
For fixed $a \ne b$, the four tuples in $\Gamma_{31}$ with repeated value
$a$ and singleton $b$ contribute $4M_{aa}M_{ab}$ (using symmetry of $M$).
Summing over all ordered pairs $a \ne b$ and using $\sum_{b \ne a}
M_{ab} = 1 - M_{aa}$,
\begin{align*}
  \mathcal{T}_{31}
  &= 4\sum_{a=1}^N M_{aa}\sum_{b \ne a} M_{ab} \\
  &= 4\sum_{a=1}^N M_{aa}(1 - M_{aa}) \\
  &= 4\bigl(\tr(M) - \Delta\bigr).
\end{align*}

\noindent\textit{Computing $\mathcal{T}_{22}$.}
For each unordered pair $\{a,b\}$ with $a \ne b$, the six tuples in
$\Gamma_{22}$ with values $\{a,a,b,b\}$ contribute $2M_{aa}M_{bb} +
4M_{ab}^2$. Summing over all such pairs,
\[
  \mathcal{T}_{22}
  = 2\sum_{a < b} M_{aa} M_{bb} + 4\sum_{a < b} M_{ab}^2.
\]
Since
\begin{align*}
  2\sum_{a < b} M_{aa} M_{bb}
  &= \tr(M)^2 - \Delta, \\
  4\sum_{a < b} M_{ab}^2
  &= 2\bigl(\tr(M^2) - \Delta\bigr),
\end{align*}
we obtain
\[
  \mathcal{T}_{22}
  = \tr(M)^2 + 2\tr(M^2) - 3\Delta.
\]

\noindent\textit{Computing $\mathcal{T}_{211}$.}
Fix pairwise distinct $a, b, c$ with $a$ the repeated value. The six
corresponding tuples in $\Gamma_{211}$ contribute $2M_{aa}M_{bc} +
4M_{ab}M_{ac}$. Thus
\[
  \mathcal{T}_{211}
  = 2\sum_{a=1}^N M_{aa}
    \sum_{\substack{b,c \ne a \\ b \ne c}} M_{bc}
  + 4\sum_{a=1}^N
    \sum_{\substack{b,c \ne a \\ b \ne c}} M_{ab} M_{ac}.
\]
For the first sum, using $\sum_{b,c} M_{bc} = N$, $\sum_b M_{ab} = 1$,
$\sum_c M_{ca} = 1$, and $\sum_{b \ne a} M_{bb} = \tr(M) - M_{aa}$,
\begin{align*}
  \sum_{\substack{b,c \ne a \\ b \ne c}} M_{bc}
  &= N - 1 - 1 + M_{aa} - \bigl(\tr(M) - M_{aa}\bigr) \\
  &= N - 2 - \tr(M) + 2M_{aa}.
\end{align*}
Therefore
\begin{align*}
  2\sum_{a=1}^N M_{aa}
  \sum_{\substack{b,c \ne a \\ b \ne c}} M_{bc}
  &= 2\sum_{a=1}^N M_{aa}\bigl(N - 2 - \tr(M) + 2M_{aa}\bigr) \\
  &= 2(N-2)\tr(M) - 2\tr(M)^2 + 4\Delta.
\end{align*}
For the second sum, since $\sum_{b \ne a} M_{ab} = 1 - M_{aa}$,
\[
  \sum_{\substack{b,c \ne a \\ b \ne c}} M_{ab} M_{ac}
  = \Bigl(\sum_{b \ne a} M_{ab}\Bigr)^2 - \sum_{b \ne a} M_{ab}^2
  = (1-M_{aa})^2 - \sum_{b \ne a} M_{ab}^2.
\]
Summing over $a$,
\begin{align*}
  \sum_{a=1}^N \sum_{\substack{b,c \ne a \\ b \ne c}} M_{ab} M_{ac}
  &= \sum_{a=1}^N\bigl(1 - 2M_{aa} + M_{aa}^2\bigr)
   - \sum_{a=1}^N\sum_{b \ne a} M_{ab}^2 \\
  &= N - 2\tr(M) + \Delta - \bigl(\tr(M^2) - \Delta\bigr) \\
  &= N - 2\tr(M) + 2\Delta - \tr(M^2).
\end{align*}
Combining both parts,
\[
  \mathcal{T}_{211}
  = 4N + (2N-12)\tr(M) - 2\tr(M)^2 - 4\tr(M^2) + 12\Delta.
\]

\noindent\textit{Computing $\mathcal{T}_{1111}$.}
By the partition identity,
\begin{align*}
  \mathcal{T}_{1111}
  &= N^2 - \mathcal{T}_4 - \mathcal{T}_{31}
     - \mathcal{T}_{22} - \mathcal{T}_{211} \\
  &= N^2 - 4N - (2N-8)\tr(M)
     + \tr(M)^2 + 2\tr(M^2) - 6\Delta.
\end{align*}

\medskip
\noindent\textit{Combining.}
Substituting moments (a)--(h) and the $\mathcal{T}_\alpha$ formulas
into~\eqref{eq:second_moment_master},
\begin{align}
  N^2\,\mathbb{E}_\pi\bigl[\sqd(f_\pi P)^2\bigr]
  &= \frac{S_4(Z)}{N}\,\mathcal{T}_4
   - \frac{S_4(Z)}{N(N-1)}\,\mathcal{T}_{31}
   + \frac{N^2\sqd(f)^2 - S_4(Z)}{N(N-1)}\,\mathcal{T}_{22}
   \notag\\
  &\quad
   + \frac{2S_4(Z) - N^2\sqd(f)^2}{N(N-1)(N-2)}\,\mathcal{T}_{211}
   + \frac{3N^2\sqd(f)^2 - 6S_4(Z)}{N(N-1)(N-2)(N-3)}\,\mathcal{T}_{1111}.
  \label{eq:second_moment_Tform}
\end{align}
We collect coefficients separately using common denominator $\mathcal{D} = (N-1)(N-2)(N-3)$.

\medskip
\noindent\textit{Coefficient of $\sqd(f)^2$.}
The $\sqd(f)^2$ contributions arise from $\mathcal{T}_{22}$,
$\mathcal{T}_{211}$, and $\mathcal{T}_{1111}$:
\begin{align*}
  &(N-2)(N-3)\,\mathcal{T}_{22}
   - (N-3)\,\mathcal{T}_{211}
   + 3\,\mathcal{T}_{1111} \\
  &\quad=
    (N^2-3N+3)\bigl(\tr(M)^2 + 2\tr(M^2)\bigr)
    - N^2 \\
  &\qquad
    - 2(N^2-6N+6)\tr(M)
    - 3N(N-1)\Delta,
\end{align*}
giving a contribution of
\[
  \frac{N\sqd(f)^2}{\mathcal{D}}
  \Bigl[
    (N^2-3N+3)\bigl(\tr(M)^2 + 2\tr(M^2)\bigr)
    - N^2
    - 2(N^2-6N+6)\tr(M)
    - 3N(N-1)\Delta
  \Bigr].
\]

\noindent\textit{Coefficient of $S_4(Z)$.}
The $S_4(Z)$ contributions arise from all five patterns:
\begin{align*}
  &(N-1)(N-2)(N-3)\,\mathcal{T}_4
   - (N-2)(N-3)\,\mathcal{T}_{31}
   - (N-2)(N-3)\,\mathcal{T}_{22} \\
  &\quad
   + 2(N-3)\,\mathcal{T}_{211}
   - 6\,\mathcal{T}_{1111} \\
  &\quad=
    -(N-1)\bigl(\tr(M)^2 + 2\tr(M^2)\bigr)
    + 2N
    - 4\tr(M)
    + N(N+1)\Delta,
\end{align*}
giving a contribution of
\[
  \frac{S_4(Z)}{N\mathcal{D}}
  \Bigl[
    -(N-1)\bigl(\tr(M)^2 + 2\tr(M^2)\bigr)
    + 2N
    - 4\tr(M)
    + N(N+1)\Delta
  \Bigr].
\]
Dividing~\eqref{eq:second_moment_Tform} by $N^2$ yields the stated formula for $\mathbb{E}_\pi[\sqd(f_\pi P)^2].$

\newpage
\end{document}